\documentclass[reqno]{amsart}
\usepackage{amsthm,amsmath,amsfonts,amssymb,mathtools,bbm,bm}

\usepackage[left=3.4cm,right=3.4cm,top=2.5cm,bottom=3cm]{geometry}
\usepackage[hidelinks]{hyperref}
\usepackage{enumerate}
\usepackage{graphicx,xcolor}
\usepackage{dsfont}

\newcommand	{\C}			{\mathbb{C}}
\newcommand	{\E}			{\mathbb{E}} 
\newcommand	{\N}			{\mathbb{N}}
\renewcommand{\P}			{\mathbb{P}}

\newcommand	{\R}			{\mathbb{R}}

\DeclareMathOperator	{\supp}			{supp}
\renewcommand{\Re}{\operatorname{Re}}
\renewcommand{\Im}{\operatorname{Im}}
\newcommand{\dd}{{\rm d}} 
\newcommand{\ton}{\overset{}{\underset{n\to\infty}\longrightarrow}}

\theoremstyle{plain}
\newtheorem{theorem}{Theorem}[section]
\newtheorem{lemma}[theorem]{Lemma}
\newtheorem{corollary}[theorem]{Corollary}
\newtheorem{prop}[theorem]{Proposition}
\newtheorem{conjecture}{Conjecture}

\theoremstyle{definition}

\newtheorem{example}[theorem]{Example}
\theoremstyle{remark}
\newtheorem{remark}[theorem]{Remark}
\newtheorem{problem}[theorem]{Open Problem}

\numberwithin{equation}{section}

\title{Heat flow and repeated differentiation of polynomials with i.i.d.~roots}
\author{Jonas Jalowy}
\address{Paderborn University, Institute of Mathematics, Warburger Str. 100, 33098 Paderborn, Germany}
\email{jjalowy@math.upb.de}
\date{\today}



\begin{document}
\maketitle
\begin{abstract}
How do the zeros of a polynomial evolve under the holomorphic heat flow or repeated differentiation? 
In this work, we develop a unified probabilistic proof to three conjectures on the evolution of polynomial zeros under holomorphic heat flow and repeated differentiation. We study these two evolutions for random polynomials with i.i.d.~roots $z_1,\dots,z_n\sim\mu_0$, where $\mu_0$ on $\C$ satisfies suitable $2+\delta$ logarithmic moment conditions. 

For the heat flow of small time $t>0$ and Lipschitz continuous Stieltjes transform of $\mu_0$, we identify the limit distribution $\mu_t$ as explicit push-forward of $\mu_0$ under a transport map. For rotationally invariant $\mu_0$, we characterize the limit after $\lfloor tn\rfloor$ differentiations through its radial quantiles. For instance, the heat flow evolves the circular law into the elliptic and semicircle law, while differentiation moves surviving mass towards the origin. We also show that $o(n)$ derivatives preserve the initial distribution $\mu_0$. In fact, all convergences hold almost surely.

The common proof strategy crucially relies on recursion identities from leaving out a root, and concentration inequalities, which lead to self-consistent equations for Stieltjes transforms. This brings a method familiar from random matrix theory to polynomial evolutions. Moreover, it allows for generalizations beyond rotational symmetric distributions and quantitative stability estimates for $o(n)$ differentiations and vanishing heat-flow times.

%

\end{abstract}

\section{Introduction}\label{sec:intro}
\subsection{Background and conjectures}\label{sec:backconj}
The history of finding the roots of a polynomial affected by a differential operator can be traced back to at least the Gauss--Lucas theorem, locating the roots of the derivative $P'$ inside the convex hull of the roots of the original polynomial $P$.
For a sequence of polynomials $P_n$ of degree $n$, a more precise \emph{open question} arises: 
\begin{quotation}
	\begin{center}	\emph{For a sequence of polynomials $(P_n)_{n\in \N}$ with limiting zero distribution $\mu_0$ on $\C$,\\ and a time-evolution operator $D_t$ on polynomials, \\what is the limiting zero distribution $\mu_t$ of $D_tP_n$ as $n \to \infty$?}
	\end{center}
\end{quotation}

Of particular interest are the evolutions of polynomial zeros under (repeated) differentiation and the (backward, holomorphic) heat flow. We will explore the development of this highly active area of research throughout this paper, and encourage the interested reader to dive into the literature for fascinating connections to random matrix theory, free probability, analytic number theory, potential theory, optimal transport and (partial) differential equations.


Let us now be more precise and introduce our setting. We consider the random polynomial
$$
P_n(z):=\prod_{j=1}^n(z-z_j)
$$
of degree $n\in\N$ with  independent and identically distributed roots $z_1,z_2,\ldots\sim \mu_0$ for some distribution $\mu_0$ on $\C$.
Define the \emph{heat evolved polynomial}
\begin{align}
	P_{t,n}(z):=\exp\left(-\frac{t}{2n}\partial_z^2\right) P_n(z),\quad t\ge 0,
\end{align}
via terminating power series and the \emph{repeatedly differentiated polynomial}
\begin{align}
	Q_{t,n}(z):=\partial_z^{\lfloor tn\rfloor}P_n(z), \quad 0\le t<1,
	\label{eq:diff_Q}
\end{align}
of degree $d_n:=n-\lfloor tn\rfloor$. Here, $\partial_z$ denotes the holomorphic (Wirtinger) derivative. Their \emph{empirical distributions of roots} are given by
\begin{align}
	\mu_{t,n}:= \frac 1 n \sum_{z:P_{t,n}(z)=0} \delta_z,\qquad 
	\nu_{t,n}:=\frac{1}{d_n}\sum_{z:Q_{t,n}(z)=0}\delta_z,
	\label{eq:nu_tn}
\end{align}
where zeros are counted with multiplicity.  
Finally, define the \textit{Stieltjes transform}
\begin{align} m_0(z)=\int_{\mathbb C}\frac{1}{z-w}\mu_0(dw),\end{align}
 for almost all $z\in\C$. It satisfies $\partial_{\bar z}m_{0}=\pi\mu_{0}$ distributionally. Denote by $T_\#\mu_0$ the push-forward of $\mu_0$ under some map $T$,  and $\Rightarrow$ weak convergence of distributions. For rotationally invariant distributions, we use the radial cumulative function $F_0(r):=\mu_0(B_r)$ for $B_r=\{z\in\C:|z|<r\}$ and write $F^{-1}(u):=\inf\{r\ge0:F(r)\ge u\}$ for its radial quantile functions.

As $n\to \infty$, we are interested in the existence and descriptions of the weak limits of $\mu_{t,n}$ and $\nu_{t,n}$. Recent works predict that these limits depend only on the initial distribution $\mu_0$ and admit explicit descriptions through push-forwards under transport maps or transported radial quantiles.
 
\begin{conjecture}[{\cite[Conj. 1.4]{hallho}}]\label{conj1}
If $m_0$ is Lipschitz and $t>0$ sufficiently small, then $\mu_{t,n}\Rightarrow\mu_t$ in probability, where $\mu_t=(T_t)_\#\mu_0$ under the transport map $T_t(w):=w+t m_0(w)$.
\end{conjecture}
This conjecture was formulated in broader a universality framework
and has so far been established for random polynomials with independent coefficients \cite{heatflowrandompoly}, recently for characteristic polynomials of Ginibre matrices \cite{assiotis2026heat}, see also \cite{GAF-paper,HJK}. For i.i.d.-rooted polynomials, it remained open prior to the present work, as did the following conjecture.
\begin{conjecture}[{\cite[Conj. 2.1]{HK21}, \cite[(1)]{OSteiner}}]\label{conj2}
	If $\mu_0$ is rotationally invariant and $0<t<1$, then $\nu_{t,n}\Rightarrow\nu_t$ in probability, with rotationally invariant limit $\nu_t$ characterized by $F_t(r):=(1-t)\nu_t(B_r)$ via its quantile $F_t^{-1}(u)=\frac{u}{u+t}F_0^{-1}(u+t)$ for $0<u<1-t$.
\end{conjecture} 
Previous results consider polynomials with independent coefficients \cite{FengYao19,COR23,diff-paper} and some special classes of deterministic polynomials \cite{BHS24,GNV25,NV26,HP}.
Furthermore, it has been predicted that any number of $o(n)$ derivatives of $P_n$ with i.i.d.~roots retains the limiting distribution $\mu_0$, which in our notation translates to the following.
\begin{conjecture}[{\cite[\S 4]{MV24}, \cite[Conj. 1.1]{ANP}}]\label{conj3}
If $t_n\to 0$, then $\nu_{t_n,n}\Rightarrow\mu_0$ almost surely.
\end{conjecture} 
Independently, this conjecture has been solved \emph{very} recently and in full generality $t_n=o(1)$ by \cite{zhu}.
Previous results have been established for $t_n=1/n$ in \cite{K15}, for $t_n=\mathcal O(1/n)$ in \cite{Byun}, for $t_n=\mathcal O(\log n/(n\log\log n))$ in \cite{MV24} and for $t_n=o(1/\log n)$ in \cite{ANP}.

\medskip

In this work, we give a \textbf{unifying proof of all three conjectures} under suitable logarithmic moment assumptions of order $2+\delta$ on the initial distribution $\mu_0$.
 In addition, we strengthen Conjecture \ref{conj1} and \ref{conj2} to almost sure convergence, see Theorem \ref{thm:main} and Theorem \ref{thm:diff_radial}, and establish quantitative versions of Conjecture \ref{conj3} as well as its heat flow analogue $\mu_{t_n,n}\Rightarrow\mu_0$ for $t_n\to0$, see Theorem \ref{thm:vanish}. We will also derive more abstract statements allowing for some  non-symmetric $\mu_0$ in Theorem \ref{thm:diff_main}, discuss non-identical distributions in Remark \ref{rem:gen}, prove local push-forward theorems such as Proposition \ref{prop:localpushforward} and stability under perturbation in Proposition \ref{prop:diff_regularization}. A primary example in Corollary \ref{cor:circ} shows that the heat flow evolves the circular law $\mu_0$ into an elliptic law until it collapses to the semicircle law at $t=1$. Our probabilistic method of proof is novel for such problems and robust for generalizations: It relies on recursion identities from "leaving one factor out", truncation and martingale concentration of Stieltjes transforms, in the same spirit as in random matrix theory, cf.~Remark \ref{rem:idea}.
 
 While finishing this manuscript,  I learned through correspondence with Sean O'Rourke of his independent work \cite{sean}, which resolves Conjecture \ref{conj2} in greater generality of only assuming finite first logarithmic moment of $\mu_0$. Our two works use substantially different methods and yield complementary results, and we agreed to release our manuscripts simultaneously. 
\subsection{Push-forward theorem under heat flow} 
The following is our first main result.
\begin{theorem}\label{thm:main}
Assume $z_j\sim \mu_0$ are i.i.d.~roots, with Stieltjes transform $m_0$ being Lipschitz continuous with constant $L>0$ and satisfies 
\begin{align}
	\int_{\C}\log^{2+\delta}(1+|w|)\dd\mu_0(w)&<\infty \label{eq:log-ass}
	\end{align} 
for some $\delta>0$. If $0\le t<1/L$, then we have almost sure weak convergence of the empirical root distribution of heat evolved polynomials,
$$\mu_{t,n}\Rightarrow\mu_t\qquad \text{as } n\to\infty.$$
The limit $\mu_t$ has bounded density and is the push-forward $\mu_t=(T_t)_\# \mu_0$ under the transport map 
\begin{align}
	T_t(w):=w+t m_0(w), \qquad w\in\C.
\end{align}
\end{theorem}
Note that we do not need to assume local logarithmic integrability, which follows for all distributions $\mu_0$ having bounded density on $\C$ and is already implied by $m_0$ being Lipschitz due to $\partial_{\bar z}m_0=\pi\mu_0$.
Theorem \ref{thm:main}  establishes Conjecture \ref{conj1}, where we discussed related results in the literature. We postpone a discussion of the transport map $T_t$ to Remark \ref{rem:transport}. A particularly interesting case is the transition from the circular law to the elliptic law. 

\begin{corollary}\label{cor:circ}
If $\mu_0:=\operatorname{Unif}(B_1)$ is the uniform distribution on the complex unit disk, having density $\frac 1 \pi \mathbf 1_{B_1 }$, then for any $0\le t\le 1$, we have a.s.~weak convergence $\mu_{t,n}\Rightarrow\mu_t$. If $t<1$, then $\mu_t$ is the uniform distribution on the ellipse $T_t(B_1)=\{x+iy:\frac{x^2}{(1+t)^2}+\frac{y^2}{(1-t)^2}<1\}$. For $t=1$, it collapses to the semicircle law having density $$
\mu_{1}(\dd x)
=\frac{1}{2\pi}\sqrt{4-x^2}\,
\mathbf 1_{[-2,2]}(x) \dd x.$$
	\end{corollary}
	 
The assumption in Theorem \ref{thm:main} of $m_0$ being globally Lipschitz can be weakened to existence of domains whose points have a unique preimage under \(T_t\), provided the resulting inverse is continuously differentiable, implying a local push-forward theorem, see Proposition \ref{prop:localpushforward}.

\begin{remark}[Idea of proof]\label{rem:idea}
	We will prove our claims in Section \ref{sec:heat flow}.  
	The central idea of proving Theorem \ref{thm:main} is to compare the heat-evolved polynomial with the one obtained after removing a single initial factor $(z-z_k)$ from $P_n$. A commutation identity describes how multiplication by this factor interacts with the heat-flow operator, yielding recursive formulas for the polynomials, their Stieltjes transforms and their logarithmic potentials.
	  More precisely, Proposition \ref{prop:crucial_rep} states that 
\begin{align}
	m_{t,n}(z)=\frac1n\sum_{k=1}^n
	\frac{1}{z-z_k-tm_{t,n}^{[k]}(z)},
\end{align}
where $m_{t,n}^{[k]}$ is the finite-$n$ Stieltjes transform with the $k$-th factor $(z-z_k)$  left out. Such recursions  should be seen as an analogue to the Schur inversion formulas used to prove the semicircle law for Wigner matrices in random matrix theory. The technical part of our proof is to verify concentration $m_{t,n}\approx\E m_{t,n}$ and stability $m_{t,n}\approx m_{t,n}^{[k]}$ of leaving one factor out. Eventually, this leads to convergence of the Stieltjes transforms $m_t$, giving the \emph{self-consistent equation}   
\begin{align}\label{eq:selfconseq}
	m_t(z) =m_0(z-tm_t(z))=m_0(T_t^{-1}(z)),\quad \text{ in } L_{\mathrm{loc}}^1.
	\end{align}
This implies a local push-forward theorem.\end{remark}

\begin{remark}[Extensions]\label{rem:gen} This approach seems versatile enough to generalize into several directions. First and foremost, it is possible to study the limiting root measures of i.i.d.-rooted polynomials under other differential operators (cf. \cite{CJ}), such as repeated differentiation, which we shall do below in detail. Second, the whole proof will work with only minor modifications, also for independent but non-identically distributed random variables $z_k\sim\mu_0^{(k)}$ having uniformly bounded densities with joint compact support, if we assume that 
$\frac 1 n \sum_{k=1}^n\mu_0^{(k)}\Rightarrow \mu_0.$ However, we decided to only comment on this generalization at crucial steps during the proof, see Remark \ref{rem:gen2} instead of overtechnicalizing the presentation.
Third, compact support of $\mu_0$ simplifies the proof a lot, making the core steps accessible. On the other hand, weakening the $2+\delta$ logarithmic moment assumption to $1+\delta$, one may still get weak convergence in probability, see again Remark \ref{rem:gen2}.
Lastly, the investigation of dropping independence of roots for the random polynomials aiming at universality shall be left for a future work, where significantly more elaborate techniques seem to be necessary. 
\end{remark}

\begin{figure}[t]
	\centering
	\includegraphics[width=0.325\linewidth,trim={100 140 100 140},clip]{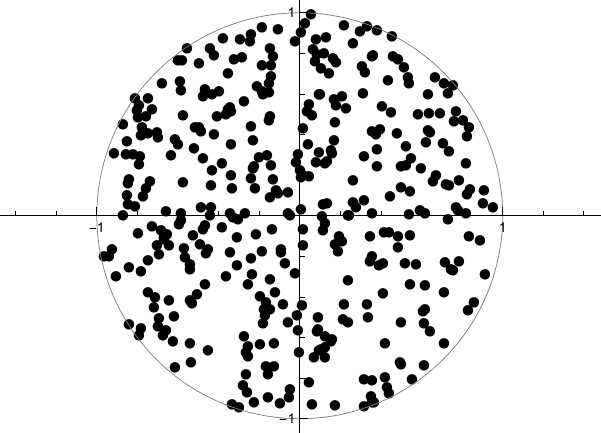}
	\includegraphics[width=0.325\linewidth,trim={100 140 100 140},clip]{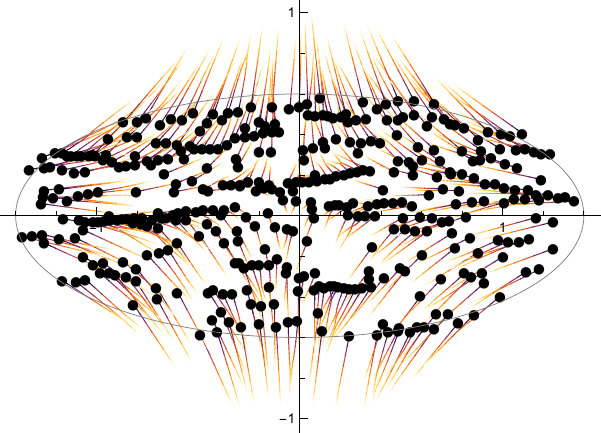}
	\includegraphics[width=0.325\linewidth,trim={100 140 100 140},clip]{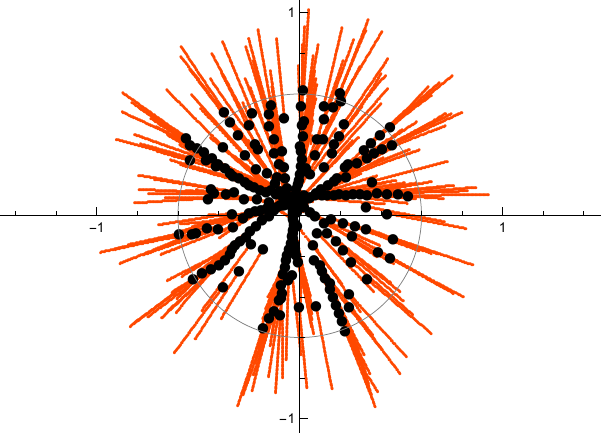}
	\caption{The black dots depict $n=400$ i.i.d.~roots $z_j\sim \mu_0$, distributed according to the circular law (left) as well as their evolutions after time $t=2/5$ under the heat flow (center) and repeated differentiation (right). The empirical distribution $\mu_{t,n}$ of roots of the heat evolved polynomial $P_{t,n}$ is (globally, macroscopically) close to the uniform distribution on the depicted ellipse, illustrating the findings of Corollary \ref{cor:circ}. The actual evolution of the roots under the heat flow is shown by orange lines, whose majority follow the trajectories of the transport map $T_t$. Similarly, the findings of Corollary \ref{cor:diff_circ} are illustrated in the right picture, whose (global, macroscopic) limit distribution $\nu_t$ of $Q_{t,n}$ is supported on the disk of radius $1-t=3/5$ and each differentiation step is shown in orange, following the trajectories of the transport map $S_t$. For a discussion on the (microscopic) line formation phenomenon, see Open Problem \ref{prob}.}
	\label{fig:FIGUREWHATELSE}
\end{figure}

\subsection{Radial push-forward theorem under repeated differentiation}
Recall the definition \eqref{eq:diff_Q} of repeatedly differentiated i.i.d.-rooted polynomial $Q_{t,n}(z):=\partial_z^{\lfloor tn\rfloor}P_n(z)$ and its empirical distribution of roots $ \nu_{t,n}$, for which we present our second main result.
\begin{theorem}\label{thm:diff_radial}
	Assume that $\mu_0$ is rotationally invariant and satisfies
	\begin{align}\label{eq:ass_rot}
	\int_{\C}|\log|w||^{2+\delta}\mu_0(\dd w)<\infty
	\end{align}
	for some $\delta>0$.
For any $0<t<1$, we have almost sure weak convergence of the empirical root distribution of repeatedly differentiated polynomials
	$$
	\nu_{t,n}\Rightarrow
	\nu_t.
	$$
	The limit $\nu_t$ is rotationally invariant, characterized by $F_t(r):=(1-t)\nu_t(B_r )$ satisfying
\begin{align}\label{eq:quantiles}
	F_t^{-1}(u)=\frac{u}{u+t}F_0^{-1}(u+t),
	\qquad 0<u<1-t.
\end{align}
It has a density on $\C$ that is locally bounded by $\frac{1}{2\pi(1-t)t s_t|z|} $ for almost all $z\neq 0$ and $s_t:=\inf\{r>0:F_0(r)>t\}$.
\end{theorem}

This establishes Conjecture \ref{conj2} via \cite[Equation (24)]{HK21} under our logarithmic moment assumption and  strengthens the statement from convergence in probability to almost surely. Recall that we discussed previous results in the literature already below Conjecture \ref{conj2}, most notably the independent work \cite{sean}, which resolves the conjecture for any rotationally invariant distribution under a weaker logarithmic moment assumption. 

Assumption \eqref{eq:ass_rot} rules out atoms at the origin, but allows mass on circles singular with respect to the Lebesgue measure $\mathcal L$ on $\C$. In the case of $\mu_0$ having a bounded density on $\C$, then \eqref{eq:ass_rot} reduces to \eqref{eq:log-ass}, and we immediately get the following corollary, in the language of Theorem \ref{thm:main}.
\begin{corollary}\label{cor:diff_radial_dens}
	Assume that $\mu_0$ is rotationally invariant with bounded density on $\C$ and satisfies \eqref{eq:log-ass}.
	 Let $0<t<1$ and define another transport map 
	\begin{align}
		S_t(w):=w-\frac{t}{m_0(w)}.
		\label{eq:diff_transport}
	\end{align}
	Then, we have almost sure weak convergence of the empirical root distribution of repeatedly differentiated polynomials
	$$
	\nu_{t,n}\Rightarrow
	\nu_t:=\frac{1}{1-t}(S_t)_\#(\mu_0|_{W_t}),
	$$
	where $W_t=\{w\in\C:\mu_0(B_{|w|} )>t\}$.
\end{corollary}

Since $\mu_0$ is rotationally symmetric, 
it is also possible to rephrase the above as a radial push-forward, which we will learn in the proof. However, for instance if $\mu_0$ is uniform on the circle $\{|z|=1\}$, then we cannot expect a push-forward Theorem, since $S_t$ would need to map each point on the circle to the continuum of radii $(0,1-t)$. The quantile function $F_0^{-1}$ being constant is allowed to split the mass, by filling this continuum with $(t,1)\ni u\mapsto 1-t/u\in (0,1-t)$.
 
 \begin{remark}[Idea of proof]\label{rem:idea2}
In our proof in Section \ref{sec:diff}, we shall replace the rotationally invariant assumption by a more abstract uniqueness condition of the corresponding self-consistent equation, see Theorem \ref{thm:diff_main} and Remark \ref{rem:nonsymmex}. The general strategy for the setting of Corollary \ref{cor:diff_radial_dens} is similar to that of the heat flow, see Remark \ref{rem:idea}. However, due to the additional restriction to the region $W_t$ of surviving mass, verifying the uniqueness criterion, and additional technicalities arising from inverting $m_{t,n}\approx 0$ in \eqref{eq:diff_transport}, we decided to address the case of repeated differentiation after the heat flow. In order to lift the push-forward statement to the setting of Theorem \ref{thm:diff_radial} without assuming densities, we use a perturbative argument, which again relies upon our recursive identities.
 Note that convergence of the Stieltjes transform outside the support has been proven in \cite{martinezfinkelshtein} for deterministic polynomials, however such exterior convergence alone does not identify the limiting measure, which is one crucial part of the proof. 
\end{remark}
Analogously to Corollary \ref{cor:circ}, let us discuss the special case of the circular law.

\begin{corollary}\label{cor:diff_circ}
	If $\mu_0=\operatorname{Unif}(B_1)$ is again the circular law, then for any $0\le t<1$, we have almost sure weak convergence $\nu_{t,n}\Rightarrow\nu_t$ with $\nu_t$ being rotationally invariant with density
	\begin{align}
		\rho_t(z)=\frac{\mathbf 1_{B_{1-t}}(z)}{2\pi(1-t)}
		\left(1+\frac{|z|^2+2t}{|z|\sqrt{|z|^2+4t}}\right)
	.
		\label{eq:diff_circ_density}
	\end{align}
\end{corollary}

\begin{remark}[Real roots and free probability]
	For real-rooted polynomials and $\mu_0$ on $\R$, deterministic analogues of all three Conjectures \ref{conj1}, \ref{conj2} and \ref{conj3} are known. The heat flow corresponds to a finite free additive convolution with a rescaled Hermite polynomial and satisfies \begin{align}\label{eq:boxplussc} \mu_{t,n}\Rightarrow\mu_0\boxplus\mathrm{sc}_t,\qquad  t\geq0,\end{align} where $\mathrm{sc}_t$ is the centered semicircle law of variance $t$, see \cite{ZakharLeeYang,VW22,JKM2,CJ}. Repeated differentiation satisfies $$\nu_{t,n}\Rightarrow \mu_0^{\boxplus 1/(1-t)}\Big(\frac{\cdot}{1-t}\Big),\qquad 0<t<1,$$ relating differentiation to fractional free additive convolution powers, see \cite{Steiner19,Steiner21,Stao,KT22,arizmendi_garza_vargas_perales,JKM1,arizmendi2026s}. These finite free convolution representations also imply (non-quantitative) vanishing-time statements in Theorem \ref{thm:vanish} below, while other differentiation regimes have been studied in \cite{CampbellAppell,CORuni}.	
	
	For general distributions on $\C$, free convolutions are ill-defined, but $\mu_t$ can still be interpreted from a free probability perspective: If $\mu_0$ is rotationally invariant and it is the Brown measure with some circular component, then Theorem \ref{thm:main} turns that component to an elliptic component, and Theorem \ref{thm:diff_radial} can be interpreted via multiplication with another $R$-diagonal operator. We refer to \cite{heatflowrandompoly,COR23,diff-paper} for more information. 
\end{remark}

\subsection{Quantitative stability under vanishing time evolutions}
Lastly, we have a (quantitative) result for $t=t_n=o(1)$ being a vanishing deterministic sequence. We use $\wedge$ to denote the minimum and the convention $0\log0=0$.
 
\begin{theorem}\label{thm:vanish}
Assume $\mu_0$ satisfies 
\begin{align}	\int_{\C}\log^{2+\delta}(1+|w|)\dd\mu_0(w) +\sup_{z\in\C}\int_{\C}|\log(1\wedge|z-w|)|^{2+\delta}\dd\mu_0(w)<\infty.\label{eq:asslocallog}
	\end{align}
 If $t_n\to 0$, then
$$\mu_{t_n,n}\Rightarrow\mu_0, \quad \text{and}\quad \nu_{t_n,n}\Rightarrow\mu_0$$
weakly almost surely. More quantitatively, if $\mu_0$ has bounded density with compact support and $\varphi\in \mathcal C^2_c$, then there exists a constant $c>0$ such that almost surely for sufficiently large $n$,
\begin{align}
	\bigg|\int\varphi\dd\mu_{ t_n,n}-\int \varphi\dd \mu_0\bigg|
	&\le c \bigg(t_n+\sqrt{\frac{\log n}n}\bigg) \label{eq:quantitativebound_mu}\\
	\bigg|\int\varphi\dd\nu_{ t_n,n}-\int \varphi\dd \mu_0\bigg|
&\le c \bigg(t_n\log\big(t_n^{-1}\big)+\sqrt{\frac{\log n}n}\bigg) 	\label{eq:quantitativebound}.
	\end{align}
\end{theorem}
In particular, this gives a rate of convergence of the empirical distributions of $\partial_z^{k_n}P_n$ for $k_n=\lfloor t_n n\rfloor=o(n)$, implying also bounds on bounded-Lipschitz metric and Kolmogorov distance, see Remark \ref{rem:rates}. For general probability distributions $\mu_0$, the qualitative statement has been proven for $k_n=1$ in \cite{K15,totik}, for $k=\mathcal O(1)$ in \cite{Byun}, for $k_n=o(n/\log n)$ in \cite{ANP} and recently for all $k_n=o(n)$ in \cite{zhu}. Under our assumption $\mu_0$ having compactly supported and bounded density, we provide a simple alternative proof of this statement, a first quantitative rate of convergence throughout $k_n=o(n)$ and the analogous result for the heat flow. 
 
For $t_n=1/n$ corresponding to $k_n=1$ differentiation, a related quantitative statement can be found in \cite[Theorem 2.16]{OW20}, where smooth statistics of $\nu_{t,n}-\nu_{0,n}$ are of order $\log n/ n$ almost surely. By the law of iterated logarithm, this implies the rate $\sqrt{\log(\log n)/n}$ to $\mu_0$, while \eqref{eq:quantitativebound} gives the near optimal rate $\sqrt{ \log n /n}$ for much higher derivative orders $k_n=\mathcal O (\sqrt{ n/\log n})$.

Regarding the extensions to non-identically distributed $z_k\sim\mu_0^{(k)}$ from Remark \ref{rem:gen}, it should be emphasized that all $\mu_0^{(k)}$ need to be assumed to be rotationally invariant for statements on $\nu_{t_n,n}$, and, that we do not claim a quantitative bound like \eqref{eq:quantitativebound} in this case.  We will see that in the case of $\mu_0$ being rotationally symmetric or having a bounded density, then \eqref{eq:asslocallog} reduces to \eqref{eq:ass_rot} or \eqref{eq:log-ass}, respectively.

\begin{remark}\label{rem:transport}
 Both transport maps
$
T_t$ and $S_t $
describe motion along straight lines with constant velocities determined by the initial distribution, for as long as the corresponding transport description remains valid. For rotationally invariant initial data, repeated differentiation moves mass radially inward following $R_t$ until it reaches the origin and "dies", cf. \cite{diff-paper,OSteiner,COR23,Steiner19}. Under the heat flow, mass is conserved and the transport goes in straight lines until it develops singularities, cf. \cite{hallho,heatflowrandompoly,GAF-paper,HJK,HK21}

This geometric picture connects several perspectives: the self-consistent equation \eqref{eq:selfconseq} yields a Hamilton–Jacobi equation for the logarithmic potential and a Burgers equation for the Stieltjes transform,
\begin{align}
\partial_tU_t=- (\partial_z U_t)^2, \quad \text{ and }\quad
\partial_tm_t=-m_t\partial_zm_t
\end{align} with holomorphic time derivatives $\partial_t$ and characteristics being the trajectories of $T_t$, cf. \cite[\S 5]{heatflowrandompoly}.
Moreover, $T_t^{-1}$ can be viewed as the subordination function for the free convolution \eqref{eq:boxplussc} with the semicircle law, see \cite[\S 6]{heatflowrandompoly}, and, $T_t$ turns out to be $2$-Wasserstein optimal for isotropic $\mu_0$ iff $\mu_0$ is a circular law. 

The vanishing-time regime also has a microscopic interpretation. A critical point $\xi_j$ near a simple root $z_j$ formally satisfies
$$
\xi_j-z_j
=-\bigg(\sum_{k\ne j}\frac1{\xi_j-z_k}\bigg)^{-1}
\approx-\frac1{n\,m_0(z_j)},
$$
see for instance \cite[Idea 1.2]{diff-paper} for more details. Since $k=1$ derivative corresponds to $t_n=1/n$, this predicts the initial velocity $-1/m_0(z_j)$. For the heat flow, simple zeros satisfy the exact system of ordinary differential equations
\begin{align}\label{eq:ODE}
\frac{d}{dt}z_j(t)
=\frac1n\sum_{k\ne j}\frac1{z_j(t)-z_k(t)}, \quad j=1,\dots,n,
\end{align}
see \cite{heatflowrandompoly,tao_blog1,tao_blog2,rodgers_tao}. Again, the initial velocity formally becomes $m_0(z_j)$ in the hydrodynamic limit. These velocities agree with $\partial_tS_t|_{t=0}$ and $\partial_tT_t|_{t=0}$, and motivate Theorem \ref{thm:vanish}. 
\end{remark}
\begin{problem}[Microscopic line formation]\label{prob}
	This work concerns the macroscopic limit $\mu_t$ of the empirical root distribution $\mu_{t,n}$. As discussed in Section \ref{sec:backconj}, the universality conjectures in \cite{hallho,heatflowrandompoly}
	predict that, under suitable assumptions, this limit depends only on the initial limiting distribution $\mu_0$ and not on the microscopic arrangement of the roots.
	
	However, it is clear from the simulations in Figure \ref{fig:FIGUREWHATELSE} that the microscopic point process changes drastically under this evolution: The \emph{line formation phenomenon} under the evolutions makes the zeros highly dependent, as it has been observed in \cite{heatflowrandompoly,HK21}. A closer inspection of the illustrations reveals that these lines form from clusters of initial roots. An extreme case of initially clustered roots is provided by high powers of a fixed polynomial, where roots coincide at finitely many locations--a setting where the limiting distributions indeed turn out to be supported on curves, see \cite{HJK} for the heat flow and \cite{BHS24} for repeated differentiation.
	
	It remains to give a rigorous explanation and microscopic description of this line formation phenomenon for i.i.d.~initial roots, under the heat flow and analogously under repeated differentiation. A heuristic local explanation can be found in \cite[Remark 2.11]{HJK}, based on \eqref{eq:ODE}.
\end{problem} 


\section*{Acknowledgements and AI-statement}
I thank Brian Hall for his question on the case $t\to 0$, which eventually led to Theorem \ref{thm:vanish} and I thank Sean O'Rourke for his cooperation when coordinating our releases. The author has been financially supported by the DFG priority program SPP 2265 Random Geometric Systems.

This work is developed by the author and human written, but AI-assisted. ChatGPT~5.6 (later 6) was used for brainstorming on the prompted approach, parts of the drafting and a final proof-check. Whenever any steps of proofs have been suggested by the AI, they have been thoroughly humanized, significantly simplified and technically detechnified\footnote{AI was used to improve this wording.}.

Funny times we are living in, where I am writing a statement to be human.

\section{Heat flow}\label{sec:heat flow}
\subsection{Recursion formulas}\label{subsec:recursion}
For any $t\ge0$ define flow operator and polynomials
$$
D_{t/n}:=\exp\left(-\frac{t}{2n}\partial_z^2\right),
\qquad P_{t,n}^{[k]}(z)=D_{t/n} \prod_{{j\le n,j \neq k}}(z-z_j)$$
and denote $P_{n}^{[k]}:=P_{0,n}^{[k]}$. We shall work with 
\emph{Stieltjes transform}  given by
\begin{align}\label{eq:m_tn}
	m_{t,n}(z):=\frac 1 n \frac{P'_{t,n}(z)}{P_{t,n}(z)}= \frac{1}{n}\sum_{k=1}^n\frac{P_{t,n}^{[k]}(z)}{P_{t,n}(z)},
	\end{align}
where the last equation follows from the product rule and commutation of $D_{t/n}$ and $\partial_z$. In the following we will use the notation $$m_{t,n}^{[k]}(z):=\frac{1}{n}\frac{P^{[k]\prime}_{t,n}(z)}
{P_{t,n}^{[k]}(z)},$$ where we emphasize that the normalization is notationally helpful, but has unconventional normalization for a Stieltjes transform.
The first crucial observation is the following representation.
\begin{prop}\label{prop:crucial_rep}
 For almost all $z$, we have
\begin{align}
	m_{t,n}(z)=\frac1n\sum_{k=1}^n
	\frac{1}{z-z_k-tm_{t,n}^{[k]}(z)}.\label{eq:crucial_rep_m}
\end{align}
\end{prop} 
This proposition follows immediately from the following control of the effect of multiplication by $z-z_j$ under the heat flow and its commutation.

\begin{lemma}[Commutation and Recursion]\label{lem:commu}
For any polynomial $f$ it holds
\begin{align*}
	D_{t/n}(zf(z))=\left(z-\frac{t}{n}\partial_z\right)D_{t/n}f.
\end{align*}
This implies the following recursion
\begin{align}\label{eq:recursion}
	P_{t,n}
 =P_{t,n}^{[k]}(z)
	\left(z-z_k-\frac{t}{n}\frac{P^{[k]\prime}_{t,n}(z)}
	{P_{t,n}^{[k]}(z)}\right), \qquad 1\le k \le n.
\end{align}
\end{lemma}
\begin{proof} Follows from a direct calculation, starting with
$$
\partial_z^{2j}(zf)=z\partial_z^{2j}f
   +2j\partial_z^{2j-1}f,\qquad j\geq1.
$$
Consequently,
\begin{align*}
D_{t/n}(zf)
&=\sum_{j=0}^{\infty}\frac{(-t/(2n))^j}{j!}
  \partial_z^{2j}(zf)\\
&=zD_{t/n}f-\frac{t}{n}\partial_z
  \sum_{j=1}^{\infty}\frac{(-t/(2n))^{j-1}}{(j-1)!}
  \partial_z^{2j-2}f\\
&=\left(z-\frac{t}{n}\partial_z\right)D_{t/n}f.
\end{align*}
as claimed. The recursion follows from 
\begin{align*}
P_{t,n}
=D_{t/n}\bigl((z-z_k)P_{n}^{[k]}\bigr) =\left(z-z_k-\frac{t}{n}\partial_z\right)P_{t,n} ^{[k]} 
=P_{t,n}^{[k]}(z)
  \left(z-z_k-\frac{t}{n}\frac{P^{[k]\prime}_{t,n}(z)}
  {P_{t,n}^{[k]}(z)}\right).
\end{align*}
\end{proof}

Besides the Stieltjes transform, also recall the definition of the \textit{logarithmic potential}
$$U_0(z)=\int_{\mathbb C}\log|z-w| \mu_0(dw), \quad z\in\C,$$
satisfying $2\partial_zU_0=m_0$  distributionally.

\begin{prop}
The  logarithmic potential  of the root distribution of the monic polynomial $P_{t,n}$  is given by
$
U_{t,n}(z):=\frac1n\log|P_{t,n}(z)|$
and Lemma \ref{lem:commu} also implies the recursive formula
\begin{align} 
	U_{t,n}(z)=\frac1n\sum_{k=1}^n
\log\left|z-\frac{t}{n}
\frac{\partial_z D_{t/n}P_{k-1}}{D_{t/n}P_{k-1}(z)}-z_k\right|.\label{eq:crucial_rep}
\end{align}
\end{prop}

\subsection{Convergence of the Stieltjes transform}
The idea of the proof is simple:
If we assume that the Stieltjes transform $m_t$ of the limit $\mu_t$ of $\mu_{t,n}$ exists, then we may expect
\begin{align}\label{eq:formal_m}
tm_{t,n}^{[k]}(z)
\approx tm_t(z)
\end{align}
and convergence of the sum \eqref{eq:crucial_rep_m} to an integral given by
\begin{align}
	m_{t,n}(z)\approx \int  \frac 1{z-w-tm_t(z)} \mu_0(\dd w)= m_0(z-tm_t(z)).\label{eq:integral_m}
\end{align}
Note that the summands in \eqref{eq:crucial_rep_m} are not themselves i.i.d., due to that term \eqref{eq:formal_m}, which on the other hand is independent of $z_k$. The following limit will turn out to imply push-forward representations of the limit distribution $\mu_t$.

\begin{prop}[Limiting self-consistent equation]\label{prop:m_t_limits}
Assume that $\mu_0$ has a bounded density and satisfies the $2+\delta$ logarithmic moment assumption \eqref{eq:log-ass}. Let $t\ge 0$ and $\mathcal D_t\subseteq \C$ be an open domain of points $z\in\mathcal D_t$ such that $T_t:T_t^{-1}(\mathcal D_t)\to \mathcal D_t$ is bijective.
We have convergence of the Stieltjes transforms
\begin{align}\label{eq:U_t_limits}
	m_{t,n}(z)\ton m_t(z):=m_0(T_t^{-1}(z))= m_0(z-tm_t(z))
\end{align}
almost surely in $L^1_{\mathrm{loc}}(\mathcal D_t)$. 
Moreover, if $t_n\to0$ is a vanishing sequence depending on $n$, then $	m_{t_n,n} \to m_0$ almost surely in $L^1_{\mathrm{loc}}$.

In particular, if $m_0$ is Lipschitz continuous with constant $L>0$ and $0\le t<1/L$, then the map $x\mapsto m_0(z-tx)$ is a contraction for every $z\in\C$, and thus, the self-consistent equation \eqref{eq:selfconseq} has a unique solution and $T_t^{-1} z$ is unique. 
\end{prop}

In particular, this proves the qualitative statement of Theorem \ref{thm:vanish} for bounded densities, and, moreover $\|m_t\|_\infty\leq\|m_0\|_\infty$.

The bijectivity condition assures that for each $z\in\mathcal D_t$ there is exactly one preimage in the set theoretic inverse $T_t^{-1}(\mathcal D_t)$.
Our proof-approach of reducing the degree of the polynomial by one as in \eqref{eq:formal_m}, to get a self-consistent equation for the Stieltjes transform is standard in random matrix theory. We refer the interested reader to e.g. \cite[\S 2.3]{BS10spectral}: After truncation of the matrix entries, the Schur-complement formula is used to control the characteristic polynomial of a matrix when a row and column is removed, which leads to convergence of the mean Stieltjes transform and the proof follows from concentration around its mean. For the proof of Proposition \ref{prop:m_t_limits}, we follow similar steps and need the following lemmas.

\begin{lemma}[Stability of leave-one-out]
	\label{lem:leave_one_out}
Assume that $\mu_0$ has bounded density. Let $K\subset\mathbb C$ be compact
	, then
	\begin{align}
		\sup_{k\leq n}
		\mathbb E
		\left\|
		m_{t,n}- m_{t,n}^{[k]}
		\right\|_{L^1(K)} 
		\longrightarrow0.
		\label{eq:lol_all_p}
	\end{align}
\end{lemma}

\begin{lemma}[Concentration of the Stieltjes transform]
	\label{lem:stieltjes_concentration}
	Assume that $\mu_0$ has a bounded density and satisfies the $2+\delta$ logarithmic moment assumption \eqref{eq:log-ass}. Fix $t\ge 0$ and compact $K\subset\mathbb C$, then the sequence $\E\mu_{t,n}$ is tight and
	\begin{align}
		\left\|
		m_{t,n}-\mathbb E m_{t,n}
		\right\| _{L^1(K)}
		\longrightarrow0,
		\label{eq:stieltjes_concentration}
	\end{align}
	$\P$-almost surely and in $L^1(\P)$. Moreover, the statement remains true if $t_n\to0$ is a vanishing sequence depending on $n$.
\end{lemma}

During the upcoming proofs, we will make use of the following well known facts.
 
\begin{lemma} 
	Let $\mu$ be a signed Borel measure on $\C$ with total variation bounded by $2$,
	and let $K\subseteq\C$ be compact. Then
	\begin{align}
		\|m_\mu\|_{L^p(K)}\leq C_{K,p},
		\qquad 0<p<2, \label{eq:mLp}
	\end{align}
where $C_{K,p}$ depends only on $K$ and $p$.	If $\int_\C\log(1+|w|)|\mu|(\dd w)<\infty$, then
	\begin{align}
		\|U_\mu\|_{L^p(K)}
		\leq C_{K,p}\left(1+\int_\C\log(1+|w|)|\mu|(\dd w)\right),
		\qquad 0<p<\infty. \label{eq:ULp}
	\end{align}
	Moreover, a sequence of such measures converges vaguely $\mu_n\to\mu$ iff $m_{\mu_n}\to m_\mu$ in $L_{\mathrm{loc}}^1$.
\end{lemma}
The signed case will be used below for $\mu_{t,n}-\E\mu_{t,n}$ and, note that if $\mu$ has compact support, then \eqref{eq:ULp} is bounded. 
In particular, Lemma \ref{lem:leave_one_out} and Lemma \ref{lem:stieltjes_concentration} continue to hold for $L^p(K)$, $0<p<2$, but we will not need this.

\begin{proof}
Regarding the bounds, observe that
\begin{align*}
	\sup_{w\in\C}\int_K|z-w|^{-p}\dd\mathcal L(z)<\infty
	&\quad 0<p<2 ,\\
	\sup_{w\in\C}\int_K
	\left|\log|z-w|-\log(1+|w|)\right|^p\dd\mathcal L(z)<\infty
	&\quad 0<p<\infty .
\end{align*}
	Indeed, choose $R\geq1$ with $K\subseteq B_R$, such that for $|w|\leq2R+1$,
	the above follow from local integrability, while for $|w|>2R+1$, both kernels are uniformly bounded on $K$.
	Jensen's inequality for $\mu/|\mu|$ and Fubini's Theorem now prove the claims \eqref{eq:mLp}  and \eqref{eq:ULp} for $p\geq1$, applying the second estimate to
	$U_\mu-\int\log(1+|w|)\,\mu(\dd w)$.
	For $0<p<1$, use
	$\|f\|_{L^p(K)}\leq\mathcal L(K)^{1/p-1}\|f\|_{L^1(K)}$. For convergence statement, note that the regularized Cauchy kernel $\varphi_\varepsilon(z)=\frac{\bar z}{|z|^2\vee\varepsilon^2}$ is continuous and vanishes at $\infty$, hence vague convergence implies $\int \varphi_\varepsilon (z-w)\dd \mu_n(w)\to\int \varphi_\varepsilon (z-w)\dd \mu(w)$. Together with the uniform bound $\|m_{\mu_n}- \int \varphi_\varepsilon (z-w)\dd \mu_n(w)\|_{L^1(K)}\lesssim \varepsilon$, we obtain $L^1_{\mathrm{loc}}$ convergence by first $n\to\infty$, then $\varepsilon\to 0$. The converse statement follows directly from the distributional equation $\mu=\frac 1 \pi \partial_{\bar z}m_\mu$.
\end{proof}

Before we turn to the proofs of Lemma \ref{lem:leave_one_out} and \ref{lem:stieltjes_concentration}, let us see them in action.

\begin{proof}[Proof of Proposition \ref{prop:m_t_limits}] 
By Lemma \ref{lem:stieltjes_concentration}, we have tightness of the sequence $\E \mu_{t,n}$ (Note that in the compactly supported setting, this follows immediately from Walsh's bound, see e.g. \cite[Theorem 5.3.1]{rahman_schmeisser}). Thus, we may work with a weakly convergent subsequence $\E \mu_{t,n}\Rightarrow\mu_t$ and $\E m_{t,n}\to m_t$ in $L^1_{\mathrm{loc}}$. For any compact $K\subset \C$ 
	\begin{align*}
		\E\|m_{t,n}^{[1]}-m_t\|_{L^1(K)}\le\E\|m_{t,n}^{[1]}-m_{t,n}\|_{L^1(K)}+\E\|m_{t,n} -\E m_{t,n}\|_{L^1(K)}+\|\E m_{t,n}-m_t\|_{L^1(K)}\to 0
	\end{align*}
by Lemma \ref{lem:leave_one_out} and Lemma \ref{lem:stieltjes_concentration}. Under our bounded density assumption, $m_0$ is bounded and uniformly continuous, hence
	\begin{align}
		\mathbb E\left\|
		m_0(\cdot-tm_{t,n}^{[1]})
		-m_0(\cdot-tm_t)
		\right\|_{L^1(K)}\to0.\label{eq:m0limit}
	\end{align}
	Note that, regarding the statement of $t_n\to 0$, this continues to hold without employing Lemma \ref{lem:leave_one_out}, since $\|t_n m_{t_n,n}^{[1]}\|_{L^1(K)}\le t_n C\to 0$.
	
	On the other hand, independence of $z_k\sim \mu_0$ and \eqref{eq:crucial_rep_m}, implies the averaged identity
	\begin{align}\label{eq:averagedm}
		\E m_{t,n}(z)=\E m_0(z-tm_{t,n}^{[1]}(z)).
	\end{align} 
	By the triangle inequality, we obtain
	\begin{align}
		\left\|
		m_t	-m_0(\cdot-tm_t)
		\right\|_{L^1(K)}\le \left\|	m_t	-\E m_{t,n}	\right\|_{L^1(K)} + \E\left\| m_0(\cdot-tm_{t,n}^{[1]} ) -m_0(\cdot-tm_t)
		\right\|_{L^1(K)}\to0.\label{eq:eintollesDreieck}
	\end{align}
	Thus, our assumption $z\in K\subset \mathcal D_t$ implies uniqueness of the subsequential limiting Stieltjes transforms 
	$$m_t(z)=m_0(z-tm_t(z))\quad\text{ on }\mathcal D_t$$
	This gives convergence
	\begin{align*}
		\E m_{t,n}(z)\to m_0(z-tm_{t}(z))=m_t(z).
	\end{align*} 
	 first in $L^1(K)$. 
	Concentration of Stieltjes transforms from Lemma \ref{lem:stieltjes_concentration} implies
	\begin{align*}
		m_{t,n}(z)\to m_0(z-tm_{t}(z))=m_t(z)
	\end{align*}
	almost surely, in $L ^1(K)$. On the probability-1-event $\cap_{K} \{m_{t,n}\to m_t\text{ in }L^1(K)\}$ for a countable sequence of compact $K\subset\mathcal D_t$ that fill $\mathcal D_t$, we then have convergence in $L_{\mathrm{loc}}^1(\mathcal D_t)$ as claimed. Note that this continues to hold for $t_n\to 0$. Regarding the last statement, assume $m_0$ is Lipschitz continuous with constant $L>0$ and $0\le t<1/L$, then the map $x\mapsto m_0(z-tx)$ is a contraction for every $z\in\C$. Thus, the self-consistent equation \eqref{eq:selfconseq} has a unique solution $m_t(z)$ and $T_t^{-1} z$ is uniquely given by $z-tm_t(z)$.
\end{proof}

\begin{remark}\label{rem:formalU}
Assuming that the whole path of Stieltjes transforms  $(m_{\alpha t})_{\alpha\le 1}$ exists, then similarly to our approach we may also expect the sum \eqref{eq:crucial_rep} to be replaceable by an integral
	\begin{align}
		U_{t,n}(z)\approx\int_0^1\int_{\mathbb C}
		\log|z-\alpha t m_{\alpha t}(z)-w| \mu_0(dw) d\alpha  
		=\int_0^1 U_0\bigl(z-\alpha t m_{\alpha t}(z)\bigr) d\alpha 
		.\label{eq:integral_U}
	\end{align}
	It seems to be possible to also rigorously show \eqref{eq:integral_U} and derive the statement of Proposition \ref{prop:m_t_limits} from it using PDE-techniques (more precisely differentiating w.r.t $t$ and $z$ and solving the resulting PDE by method of characteristics). But since \eqref{eq:integral_U} needs existence of the whole path $(m_{\alpha t})_{\alpha \le 1}$ and would need additional PDE-arguments afterwards, we preferred to directly work with \eqref{eq:integral_m} instead. 
\end{remark}

\subsection{Proofs of the Lemmas}

\begin{proof}[Proof of Lemma \ref{lem:leave_one_out}]
We first claim that for any compact $K\subseteq \C$ and $1/2<p_0<1$ it holds
\begin{align}
	\mathbb E\Big[\int_K
	\left|
	m_{t,n}- m_{t,n}^{[k]}
	\right|^{p_0}\dd \mathcal L\Big]
	\leq
	C_{K,p_0,t}
	\left(
	n^{-p_0}+n^{1-2p_0}
	\right)\to 0
	\label{eq:lol_p0_estimate},
\end{align}
where $\mathcal L$ is the Lebesgue measure on $\C$.
To this end, use Lemma \ref{lem:commu} and  obtain  the identity
\begin{align}
m_{t,n}(z)- m_{t,n}^{[k]}(z) =
\frac{
	1-t\partial_z m_{t,n}^{[k]}(z)
}{
	n\left(
	z-z_k-t m_{t,n}^{[k]}(z)
	\right)
}.
\label{eq:lol_difference}
\end{align}	
Choose $p_0<1$, then by subadditivity of $x\mapsto x^{p_0}$ we shall control both summands of the numerator separately.
 For the $1$-term, observe that $z_k\sim\mu_0$ is independent of $m_{t,n}^{[k]}$ with bounded density, hence for any $p<2$
\begin{align}
	\mathbb E\Big[\int_K
	\left| 
		n\left(
		z-z_k-t m_{t,n}^{[k]}(z)
		\right) 
	\right|^{-p} \dd \mathcal L(z)\Big]
	\lesssim n^{-p} \sup_{a\in\C}
	\int_{\C}|a-w|^{-p} \mu_0(dw)
	\lesssim n^{-p},\label{eq:Lploc_Stieltjes}
\end{align}
where here and throughout, we use $\lesssim$ to denote inequality up to positive constants.
Similarly, freezing the roots $z_j, j\neq k$ first, we have	again by subadditivity for $p_0<1$,
\begin{align}
	\mathbb E\left[
	\left|
	\frac{-t\partial_z m_{t,n}^{[k]}(z)
	}{
		n\left(
		z-z_k-t m_{t,n}^{[k]}(z)
		\right)
	}
	\right|^{p_0}
	\Big|  z_j,j\neq k
	\right]
	\lesssim
	 n^{-p_0}
	\left| \partial_z m_{t,n}^{[k]}(z)
	\right|^{p_0}\le  n^{-2p_0}\sum_{\xi:P_{t,n}^{[k]}(\xi)=0} 
	|z-\xi|^{-2p_0}.\label{eq:someLpbound}
\end{align}
Integrating all summands over $z\in K$ and using \eqref{eq:mLp}, we obtain that \eqref{eq:someLpbound} is further bounded by
\begin{align}
	\mathbb E\left[ \int_K	   n^{-2p_0}\sum_{\xi:P_{t,n}^{[k]}(\xi)=0} 
	|z-\xi|^{-2p_0}\dd \mathcal L(z)\right]\lesssim n^{1-2p_0}.
\end{align}
Choosing $1/2<p_0<1$, we end up with \eqref{eq:lol_p0_estimate}. 
	The upgrade of the conclusion to $p=1$ follows from H\"older's inequality 
	on the product space $\mathcal L\times\P$ and \eqref{eq:mLp} giving
\begin{align}
	\E\int_K  |	m_{t,n}- m_{t,n}^{[k]}|\dd \mathcal L
	\le
	\left(\mathbb E\int_K |	m_{t,n}- m_{t,n}^{[k]}|^{p_0}\dd\mathcal L\right)^{\frac{q-1}{q-p_0}}
	\left(\mathbb E\int_K  |	m_{t,n}- m_{t,n}^{[k]}|^q\dd\mathcal L\right)^{\frac{1-p_0}{q-p_0}}
	\to 0\label{eq:Hodor}
\end{align}
	for $0<p_0<1<q<2$.  
\end{proof}

For the proof of Lemma \ref{lem:stieltjes_concentration}, we will first rely on an exponential concentration for the logarithmic potential, which is easier due to its higher integrability.

\begin{lemma}
	\label{lem:potential_concentration}
	Assume that $\mu_0$ has a bounded density and satisfies the $2+\delta$ logarithmic moment assumption \eqref{eq:log-ass}.	
	Then, for each fixed $t\geq0$, the sequence $\E\mu_{t,n}$ is tight and, almost surely,
	\begin{align}
		U_{t,n}-\mathbb E U_{t,n}
		\longrightarrow0
		\qquad
		\text{in }L^1_{\mathrm{loc}}(\mathbb C)
		\label{eq:potential_as_convergence}.
	\end{align}
 Moreover, the statement remains true if $t_n\to0$ is a vanishing sequence depending on $n$. If additionally, $\mu_0$ is compactly supported and  $K\subseteq\C$ is compact, then there exists $c>0$ such that, for any deterministic sequence $t_n\geq0$, 	almost surely for sufficiently large $n$,
	\begin{align}\label{eq:Uquanti}
	\|U_{t_n,n}-\E U_{t_n,n}\|_{L^1(K)}
	\leq c\sqrt{\frac{\log n}{n}}.
\end{align}
\end{lemma}

\begin{proof}Let us first consider the case of compact support $\supp\mu_0\subseteq B_R$ for some $R>1$, and later remove this assumption by truncation.
Thus, $z_k\sim \mu_0$ are bounded by $R$ with bounded density and we first claim that
	\begin{align}
		\sup_{w\in\mathbb C}
		\mathbb E\left[
		\left|
		\log|w-z_k|-U_0(w)
		\right|^q\right]
		\leq q!C_R^q,
		\qquad C_R:=C(1+\log R), \quad q\in\N_{\ge 2}
		\label{eq:uniform_log_moments}
	\end{align}
for some $C$ depending only on the uniform density bound of $\mu_0$. 
	To check this, recenter by $\log(1+|w|)$ instead of $U_0(w)=\mathbb E\log|w-z_k|$ and use Jensen's inequality 
	\begin{align*}
		\sup_{w\in\C}\E\left|\log|w-z_k|-U_0(w)\right|^q
		\le 2^q\sup_{w\in\C}
		\E\left|\log|w-z_k|-\log(1+|w|)\right|^q.
		\end{align*}
Now, observe the elementary inequality 
\begin{align} 
	\left|\log|w-z|-\log(1+|w|)\right|	\le	\log(2+|z|)	+|\log(1\wedge|w-z|)|,\label{eq:somelogineq}
	\end{align}
		which follows from separating the cases $|w-z|\lessgtr 1$.  Applying this and the local bound $\P(|w-z_k|<e^{-x})\le Ce^{-2x}$ coming from bounded density, the claim \eqref{eq:uniform_log_moments} follows from
		\begin{align*} 
		\sup_{w\in\C}\E\left|\log|w-z_k|-U_0(w)\right|^q	\le 4^q\left(
		\log^q(2+R)+Cq\int_0^\infty x^{q-1}e^{-2x}\dd x
		\right) =4^q\left(\log^q(2+R)+\frac{Cq!}{2^q}\right).
	\end{align*}
 
Now, we aim for a martingale concentration inequality. Define the filtration $\mathcal F_k:=\sigma(z_1,\ldots,z_k)$, and, the Doob martingale $M_k:=\mathbb E\left[U_{t,n}(z)\mid \mathcal F_k\right]$. It is well defined for almost all $z\in\C$ by Fubini and \eqref{eq:ULp}. By the recursion \eqref{eq:recursion}, its increments are given by
\begin{align}
M_k-M_{k-1}&=  \E\left[\frac 1 n
\log\left|z-t m_{t,n}^{[k]}(z)-z_k\right|\Big|\mathcal F_ {k}\right]-\E\left[\frac 1n
\log\left|z-t m_{t,n}^{[k]}(z)-z_k\right|\Big|\mathcal F_ {k-1}\right]\nonumber \\
&=\frac 1 n\E\left[
\log\left|z-t m_{t,n}^{[k]}(z)-z_k\right|-U_0\left(z-t m_{t,n}^{[k]}(z)\right)\Big|\mathcal F_ {k}\right]\label{eq:Doob_increment}
\end{align}
An application of the conditional Jensen's inequality gives for every $q\geq2$,
	\begin{align*}
		\mathbb E\left[
		|M_k-M_{k-1}|^q\mid\mathcal F_{k-1}
		\right]
		\le \frac{1}{n^q}  \E\left[  \left| 
		\log\left|z-t m_{t,n}^{[k]}(z)-z_k\right|-U_0\left(z-t m_{t,n}^{[k]}(z)\right) \right|^q\mid\mathcal F_{k-1}\right]
		\le		
		q!\left(\frac{C_R}{n}\right)^q, 
	\end{align*}
	where the last step follows from independence and \eqref{eq:uniform_log_moments}.
This is a conditional Bernstein moment condition for the increments of the martingale $M_k$, and we now concentrate on a concentration inequality. For any $0<\lambda < n/C_R$
\begin{align}
\E\left[e^{\lambda(M_k-M_{k-1})}\mid\mathcal F_{k-1} \right]\le 1+\sum_{q=2}^\infty \left( \frac{\lambda C_R}n\right)^q \le \exp \left( \frac{\left( \frac{\lambda C_R}n\right)^2}{1- \frac{\lambda C_R}n}\right)= \exp \left( \frac{  \lambda^2 C_R^2 }{n^2-  n\lambda C_R}\right).
\end{align}
Iterating $n$ times and using the tower property gives an exponential moment
\begin{align}
	\E\left[e^{\lambda (M_n -M_0)} \right]\le \exp \left( \frac{  \lambda^2 C_R^2 }{n-  \lambda C_R}\right).
\end{align}
Thus, the same holds for its negative and Chernoff's bound for $\lambda=\frac{nu}{2C_R^2+C_Ru}$ shows
	$$
	\mathbb P\left(
	\left|
		U_{t,n}(z)-\mathbb E U_{t,n}(z)
	\right|>u
	\right)
	\leq 
	2\exp\left(-\lambda u  +\frac{  \lambda^2 C_R^2 }{n-  \lambda C_R}	\right)=2\exp\left(- n\frac{u^2}{2(2C_R^2+C_Ru)}	\right).
	$$
	Thus, there exists a constant  $c>0$, depending only on $\mu_0$, such
	that, for almost all $z\in\mathbb C$, every $u>0$,
	every $n\geq1$, and every $t\geq0$,
	\begin{align}
		\mathbb P\left(
		\left|
		U_{t,n}(z)-\mathbb E U_{t,n}(z)
		\right|>u
		\right)
		\leq
		2\exp\left(
		-cn\min\{u^2,u\}
		\right).
		\label{eq:potential_concentration}
	\end{align}
	Hence, for every fixed $z$ by Borel--Cantelli,
	$$
	U_{t,n}(z)-\mathbb E U_{t,n}(z)\longrightarrow0
	\qquad\text{almost surely.}
	$$
Fubini's theorem implies that, almost surely, this convergence holds for Lebesgue-almost every $z$. 
 It is well known since \cite{walsh} that the roots of $P_{t,n}$ lie within distance $(2+\varepsilon)\sqrt t$ of $\supp \mu_0$, see \cite[Theorem 5.3.1]{rahman_schmeisser} and also \cite{MSS15}, \cite[Lemma 3.12]{HJK}. Hence, $\mu_{t,n}-\E\mu_{t,n}$ has uniformly bounded support and is tight. Together with uniform local $L^p$ integrability  \eqref{eq:ULp}, it follows \eqref{eq:potential_as_convergence}. 
 
Let us now show the quantitative bound \eqref{eq:Uquanti}. This turns out to be a joint work of Johan Jensen and Guido Fubini,  who show
 
 \begin{align*}
 	\E\exp\left(
 	\lambda\|U_{t_n,n}-\E U_{t_n,n}\|_{L^1(K)}
 	\right)
 	&\leq \frac1{\mathcal L(K)}
 	\int_K \E\exp\big(
 	\lambda\mathcal L(K)
 	|U_{t_n,n}(z)-\E U_{t_n,n}(z)|
 	\big)\dd\mathcal L(z)\\
 	&\leq 2\exp\left(
 	\frac{2C_R^2\lambda^2\mathcal L(K)^2}{n}
 	\right)
 \end{align*}
 for $0<\lambda<n/(2C_R\mathcal L(K))$, with $C_R$ independent of $t_n$.
 Together with Herman Chernoff they choose
$ \lambda=\frac{\sqrt{n\log n}}{C_R\mathcal L(K)}$
 and get
\begin{align}
 \P\left(
 \|U_{t_n,n}-\E U_{t_n,n}\|_{L^1(K)}
 >4C_R\mathcal L(K)\sqrt{\frac{\log n}{n}}
 \right)\leq 2n^{-2}\label{eq:2or3}
\end{align}
for sufficiently large $n$.  Finally, \'Emile Borel and Francesco Cantelli join in to conclude \eqref{eq:Uquanti} with $c=4C_R\mathcal L(K)$.

It remains to remove the compact-support assumption from the qualitative claims. We resample roots that are too large and control the effect through our recursive formulas.  Independently for each $j$, retain $z_j$ if $|z_j|<R$ and otherwise replace it by an independent sample  $z_j^{(R)}$ from the truncation 
\begin{align}
\mu_0^{(R)}:=\frac{\mu_0|_{B_R}}{\mu_0(B_R)}\label{eq:truncation}
\end{align}
for $R>1$ sufficiently large. The resulting roots $z_j^{(R)}$ are i.i.d.\ with law $\mu_0^{(R)}$ being compactly supported and uniformly bounded due to $\mu_0^{(R)}\le 2\mu_0$.

Fix $t\ge0$ and a compact $K\subseteq\C$. Replace the roots one at a time, and denote by $U_{t,n}^{(R,k)}$ the potential after replacing the first $k$ roots. By the recursion \eqref{eq:recursion},
\begin{align}
U_{t,n}^{(R,k-1)}(z)-U_{t,n}^{(R,k)}(z)
=\frac1n\left(\log|a-z_k|-\log|a-z_k^{(R)}|\right),\label{eq:onesteptruncation}
\end{align}
	where $a=z-t\tilde m_{t,n}^{[k]}(z)$ is independent of the replaced pair $(z_k,z_k^{(R)})$ and $\tilde m_{t,n}^{[k]}(z)$ belongs to $k-1$ many $z_j$'s already been truncated. 
Conditioning on $(z_j,z_j^{(R)}),j\neq k$, summing over the replacements and Cauchy-Schwarz inequality gives  
	\begin{align}
	\E\|U_{t,n}-U_{t,n}^{(R)}\|_{L^1(K)}
	&\leq \mathcal L(K)\sup_{a\in\C}
	\E\left[
	\mathbf 1_{B_R^c}(z_1)
	\left|\log|a-z_1|-\log|a-z_1^{(R)}|\right|
	\right]\label{eq:logbound} \\
	&\lesssim \mathcal L(K)\sqrt{\mu_0(B_R^c) }
	\sup_{a\in\C}\left(
	\E\left|\log|a-z_1|-\log(1+|a|)\right|^{2 }
	+\E\left|\log|a-z_1^{(R)}|-\log(1+|a|)\right|^{2 }
	\right)^{\frac1{2 }}\nonumber\\
	&\lesssim \mathcal L(K)\sqrt{\mu_0(B_R^c) }
	\to0,\label{eq:Rbound}
\end{align}
where we again used \eqref{eq:somelogineq} in the last step as well as our finite logarithmic moment assumption \eqref{eq:log-ass}. In particular the constant in \eqref{eq:Rbound} is independent of $n,t,R,K$ and $\E\|U_{t,n}\|_{L^1(K)}<\infty$. 

Walsh’s bound \cite[Theorem 5.3.1]{rahman_schmeisser} again gives $\supp \mu_{t,n}^{(R)}\subseteq B_{R+3\sqrt t}\subset B_{2R}$, hence the claimed tightness follows from testing against  $\varphi_R(z)=\varphi(z/2R)$ for a smooth cutoff $\varphi$ between $B_1$ and $B_2$
\begin{align}
	\sup_n\E\mu_{t,n}(B_{4R}^c) \leq \frac{\|\Delta\varphi_R\|_\infty}{2\pi}
	\sup_n\E\|U_{t,n}^{(R)}-U_{t,n}\|_{L^1(B_{4R})} \lesssim R^{-2}\mathcal L(B_{4R})
	\sqrt{\mu_0(B_R^c)}\to0.\label{eq:tightnessEmu}
\end{align}

Now choose $R_n:=\exp (n^{1/(2+\delta)} )$ such that $C_{R_n}\sim C n^{1/(2+\delta)}$.
The density bounds of $\mu_0^{(R_n)}$ are uniform in $n$, so \eqref{eq:2or3} with $R=R_n$  reads
$$
\P\left(
\|U_{t,n}^{(R_n)}-\E U_{t,n}^{(R_n)}\|_{L^1(K)}
>4C_{R_n}\mathcal L(K)\sqrt{\frac{\log n}{n}}
\right)\le2n^{-2},
$$
hence, using $C_{R_n}\sqrt{ {\log n}/{n}}\sim Cn^{-\delta/(4+2\delta)}\sqrt{\log n}\to 0$ for $\delta>0$ and Borel--Cantelli, we get almost surely 
\begin{align}\label{eq:UtnRn}
\|U_{t,n}^{(R_n)}-\E U_{t,n}^{(R_n)}\|_{L^1(K)}
\to 0 .
\end{align}
On the other hand,  none of the $n$ roots are actually truncated and we have $U_{t,n}^{(R_n)}=U_{t,n}$ almost surely for sufficiently large $n$. Indeed, the log-moment assumption \eqref{eq:log-ass} implies
$$
\sum_{j=1}^{\infty}
\P\left(\log^{2+\delta}(1+|z_j|)>j\right)\le \E \big( \log^{2+\delta}(1+|z_1|)\big)<\infty,
$$
hence almost surely, $|z_j|<R_j$ for all sufficiently
large $j$ and finitely many remaining roots also lie in $B_{R_n}$
for sufficiently large $n$. For such $n$, it follows from Jensen's inequality, \eqref{eq:Rbound} and \eqref{eq:UtnRn} that
$$
\begin{aligned}
	\|U_{t,n}-\E U_{t,n}\|_{L^1(K)}
	&\le
	\|U_{t,n}^{(R_n)}-\E U_{t,n}^{(R_n)}\|_{L^1(K)}
	+\|\E U_{t,n}^{(R_n)}-\E U_{t,n}\|_{L^1(K)}\\
	&\lesssim
	\|U_{t,n}^{(R_n)}-\E U_{t,n}^{(R_n)}\|_{L^1(K)}
	+ \mathcal L(K)\sqrt{\mu_0(B_{R_n}^c)}
	\to 0.
\end{aligned}
$$
Taking again a countable sequence $\overline B_r$ by compact disks proves the claimed almost sure $L_{\mathrm{loc}}^1$ convergence in \eqref{eq:potential_as_convergence}. Note that these arguments are independent of $t$, so for any deterministic sequence $t_n\to0$, it follows also $U_{t_n,n}-\E U_{t_n,n}\to 0 $ almost surely in $L^1_{\mathrm{loc}}(\C)$. 
\end{proof}

It should be remarked that the qualitative concentration claim may also follow from some Efron--Stein inequality, whose setting is well tailored to our resampling strategy.

\begin{remark}[Extensions, continued]\label{rem:gen2}
		Of course, by exchangeability, the supremum of Lemma \ref{lem:leave_one_out} can be replaced by $k=1$. But for non-identically distributed $z_k\sim \mu^{(k)}$ as in Remark \ref{rem:gen}, this plays a role. Similarly, \eqref{eq:Doob_increment} would need to replace $U_0$ by $U_0^{(k)}$ of $\mu^{(k)}$, but the bound stays the same under the assumptions of  Remark \ref{rem:gen}.
Similarly, the proof of Proposition \ref{prop:m_t_limits} carries over after replacing \eqref{eq:averagedm} by
$$ \mathbb E m_{t,n}(z) =\frac1n\sum_{k=1}^n \mathbb E\,m_0^{(k)} \bigl(z-tm_{t,n}^{[k]}(z)\bigr), $$
and using $\frac1n\sum_{k=1}^n m_0^{(k)} \to m_0$.

For a weaker $1+\delta$ logarithmic moment assumption, one may still obtain convergences in probability: Running through the above arguments again, only replacing \eqref{eq:Rbound} by H\"older's inequality and choosing $R_n=\exp(n^{1/4})$. For a different proof in probability under solely the first logarithmic moment condition, we refer to \cite{sean}.
\end{remark}

\begin{proof}[Proof of Lemma \ref{lem:stieltjes_concentration}]
	By Lemma~\ref{lem:potential_concentration} we have $U_{t,n}-\mathbb E U_{t,n}
	\to0$ almost surely in $L^1_{\mathrm{loc}}(\C)$ which implies vague convergence of $\mu_{t,n}-\E\mu_{t,n}\to0$ by the distributional Poisson equation $\Delta U_\mu=2\pi \mu$. 
Analogously, almost sure vague convergence implies $m_{t,n}-\E m_{t,n}\to 0$ in $L^1_{\mathrm{loc}}$, since $\mu_{t,n}-\E\mu_{t,n}$ has bounded total variation (by $2$).  The convergence in expectation follows from dominated convergence and the uniform bound \eqref{eq:mLp}.
\end{proof}

\subsection{The (local) push-forward.}

This is all we need to prove a pushforward theorem by standard arguments.

\begin{prop}[Local push-forward Theorem]\label{prop:localpushforward}
Let $t\ge 0$ and $\mathcal D_t\subseteq \C$ be a domain of regular points $z\in\mathcal D_t$ such that $T_t:T_t^{-1}(\mathcal D_t)\to \mathcal D_t$ is a $\mathcal C^1$-diffeomorphism, in particular it is a bijection. Then, the densities of the common restrictions of all subsequential limits $\mu_t|_{\mathcal D_t}$ and the pushforward $(T_t)_\#\mu_0$ coincide locally on $\mathcal D_t$.
\end{prop}
\begin{proof}The statement is trivial for $t=0$, so assume $t>0$.
We follow the approach of \cite{heatflowrandompoly} and rephrase the Jacobian $JT_t^{-1}$ of $T_t^{-1}$ as the inverse Jacobian of $T_t$. Since $T_t$ is locally differentiable, so is $m_0$. By the inverse function theorem
\begin{align*}
	\left(
	\begin{array}
		[c]{cc}%
		\partial_z T_t^{-1}(z) & \partial_{\bar{z}} T_t^{-1}(z)\\
		\partial_z\overline{T_t^{-1}(z)} & \partial_{\bar z}\overline{T_t^{-1}(z)}%
	\end{array}
	\right)
	=JT_t^{-1}(z)
	=JT_t(w)^{-1}
	&=\left(
	\begin{array}
		[c]{cc}%
		\partial_w T_t(w) & \partial_{\bar{w}} T_t(w)\\
		\partial_w\overline{T_t(w)} & \partial_{\bar w}\overline{T_t(w)}%
	\end{array}
	\right)^{-1}\\
	&=\frac{1}{\det JT_t(w)}\left(
	\begin{array}
		[c]{cc}%
		\partial_{\bar w} \overline{T_t(w)} &- \partial_{\bar{w}} T_t(w)\\ -
		\partial_w\overline{T_t(w)} & \partial_{ w}T_t(w)%
	\end{array}
	\right)
\end{align*}
The density of $\mu_t|_{\mathcal D_t}$ exists and is given by $\pi p_t(z)=\partial_{\bar z}m_t(z)\ge 0$. By the chain rule we have
\begin{align}
	\partial_{\bar{z}}m_{t}(z)
	&=\partial_{\bar{z}}\big(m_{0}(T_t^{-1}(z))\big)=(\partial_z m_0)(T_t^{-1}(z))\cdot\partial_{\bar z}T_t^{-1}(z) + (\partial_{\bar z}m_0)(T_t^{-1}(z))\cdot\partial_{\bar z}\overline{T_t^{-1}}(z)
	\\
	&=\frac{1}{\det JT_t(w)}\Big(-(\partial_w m_0)(w)\cdot\partial_{\bar w}T_t(w) + (\partial_{\bar w}m_0)(w)\cdot\partial_{w}T_t(w)\Big).\label{eq:Jacobiancalc}
\end{align}
But since $T_{t}w=w+tm_{0}(w)$, we have $\partial_{\bar w}T_t(w)=t\partial_{\bar w} m_0(w)$ and $\partial_w T_t(w)=1+t \partial_w m_0(w)$
giving
\begin{align}
	\partial_{\bar{z}}m_{t}(z)
	=\frac{\partial_{\bar w } m_0(w)}{\det JT_t(w)}\ge 0.\label{eq:densgross}
\end{align}
In particular, the determinant is positive. After dividing this result by $\pi$, the claim follows by the change of
variables theorem: the density of $(T_{t})_{\#}\mu_{0}$ at $z=T_t(w)$ is just
$1/{|\det JT_t(w)|}$ times the density of $\mu_{0}$ at $w$. \footnote{It is possible to extend the above proof to weak derivatives of $m$, weakening the assumption of $T_t$ having differentiable inverse.}
\end{proof}

\begin{proof}[Proof of Theorem \ref{thm:main}]
It follows from Proposition \ref{prop:m_t_limits} that 	$m_{t,n}\to m_t$ almost surely in $L^1_{\mathrm{loc}}(\C)$, hence we have vague convergence $\mu_{t,n}\to\mu_t=\frac 1 \pi \partial_{\bar z}m_t$ almost surely. 
If $T_t$ would already be a diffeomorphism, then the push forward follows from Proposition \ref{prop:localpushforward} for $T_t^{-1}(\mathcal D_t)$ having full mass under $\mu_0$. 
But $T_t$ is already globally bi-Lipschitz, with Lipschitz constant $(1-tL)^{-1}$ of $T_t^{-1}$, since $m_0$ is Lipschitz. The self-consistent equation \eqref{eq:selfconseq} holding almost everywhere lifts to a Lipschitz representative $ m_t(z)=m_0(T_t^{-1}(z))$, thus $\mu_t$ has bounded density $\frac 1 \pi \partial_{\bar z}m_t\le \frac{L}{\pi(1-t L)}$ almost everywhere. 
By Rademacher’s theorem, $T_t$ is almost everywhere differentiable such that the calculation in Proposition \ref{prop:localpushforward} remains valid almost everywhere and gives \eqref{eq:densgross} for almost every $w$, i.e. the densities satisfy $p_t(T_t(w))\det JT_t(w)=p_0(w)$ almost everywhere giving $\mu_t=(T_t)_\#\mu_0$. Since $\mu_t$ is a probability measure, the almost sure vague convergence is weak convergence. \end{proof}

\begin{proof}[Proof of Corollary \ref{cor:circ}]
The statement for $t<1$ follows from 
\begin{align}\label{eq:mSC}
m_0(w)=
\begin{cases}
	\overline w,& |w|\le1,\\
	1/w,& |w|>1,
\end{cases}
\end{align}
which gives $T_t (w)=(1+t)\Re(w)+i(1-t)\Im(w) $ for $|w|\le 1$. This makes $m_0$ globally $1$-Lipschitz and $T_t$ a $\mathcal C^1$ diffeomorphism between $B_1 $ and the given ellipse $\mathcal D_t=T_t(B_1)$.
	 
Let us now consider $t=1$. By the compactness and concentration arguments of
	Proposition~\ref{prop:m_t_limits}, it suffices to identify every
	subsequential limiting Stieltjes transform $m_1$, which satisfies
	$$
	m_1(z)=m_0(z-m_1(z))
	\qquad\text{for almost every }z\in\C.
	$$
	These arguments remain valid at $t=1$; the strict inequality $tL<1$ was used only for uniqueness, for which standard arguments from the semicircle law apply, such as in \cite[\S4,\S5]{HJK}.
	 
	Fix $z\notin\R$ where the self-consistent equation holds and put
	$w=z-m_1(z)$ for some subsequential limit $m_1(z)$. If $|w|\le1$, then
	$z=w+\overline w\in\R$, a contradiction. Hence \eqref{eq:mSC} implies $m_1(z)=1/(z-m_1(z))$ and $|m_1(z)|<1$. However, the roots of $m^2-zm+1=0$ have product $1$, and neither has
	modulus $1$ when $z\notin\R$. Thus, $|m_1(z)|<1$ uniquely selects the branch
	$$
	m_1(z)=\frac{z-\sqrt{z^2-4}}2,
	$$
	where $\sqrt{z^2-4}\sim z$ at infinity. This uniquely identifies the Stieltjes transform of the semicircle law and the claimed almost sure convergence follows.
\end{proof} 

\begin{proof}[Proof of Theorem \ref{thm:vanish} for $\mu_{t,n}$] 
The qualitative statement $\mu_{t_n,n}\Rightarrow\mu_0$ for $\mu_0$ having bounded density follows from Proposition \ref{prop:m_t_limits}. We will see how to remove the density assumption in Proposition \ref{prop:diff_regularization} below. Let us now prove the quantitative claim \eqref{eq:quantitativebound_mu}.

We first bound $\|\E U_{t_n,n}-U_0\|_{L^1(K)}$ for any compact
$K\subseteq\C$. Conditioning \eqref{eq:crucial_rep} on
$z_1,\ldots,z_{k-1}$ gives
$$
\E U_{t_n,n}(z)-U_0(z)
=\frac1n\sum_{k=1}^n
\E\left[
U_0\left(z-\frac{t_n}{n}
\frac{\partial_z(D_{t_n/n}P_{k-1})(z)}
{(D_{t_n/n}P_{k-1})(z)}\right)-U_0(z)
\right].
$$
Bounded density and compact support imply that $U_0$ is globally Lipschitz with constant $\|m_0\|_\infty<\infty$.
Consequently, Fubini's theorem and \eqref{eq:mLp} give
\begin{align*}
	\|\E U_{t_n,n}-U_0\|_{L^1(K)}
 \leq \frac{t_n\|m_0\|_\infty}{n^2}
	\sum_{k=1}^n \E\int_K
	\left|
	\frac{\partial_z(D_{t_n/n}P_{k-1})(z)}
	{(D_{t_n/n}P_{k-1})(z)}
	\right|\dd\mathcal L(z) \lesssim \frac{t_n}{n^2}\sum_{k=1}^n(k-1)
	\lesssim t_n,
\end{align*}
where $k-1$ comes from the degree of $D_{t_n/n}P_{k-1}$. Together with \eqref{eq:Uquanti}, this gives a deterministic $C_K>0$, almost surely
for sufficiently large $n$,
\begin{align}\label{eq:quant_U_bound}
\|U_{t_n,n}-U_0\|_{L^1(K)}
\le C_K\bigg(t_n+\sqrt{\frac{\log n}{n}}\bigg).
\end{align}
Fix some $\varphi\in\mathcal C_c^2(\C)$ and choose
$K\supseteq\operatorname{supp}\varphi$, then the distributional Poisson equation gives
$$ 
	\left|\int\varphi\dd\mu_{t_n,n}
	-\int\varphi\dd\mu_0\right|
	=\frac1{2\pi}
	\left|\int_K\Delta\varphi\,(U_{t_n,n}-U_0)\dd\mathcal L\right|\leq \frac{C_K}{2\pi}\|\Delta\varphi\|_\infty
	\left(t_n+\sqrt{\frac{\log n}{n}}\right).
$$
This proves the first claim \eqref{eq:quantitativebound_mu}, while the second bound \eqref{eq:quantitativebound} will be proven in Section \ref{sec:weakling}.
\end{proof}

\begin{remark}\label{rem:rates}
For a bound that is uniform in a class of $\varphi$, one may use smoothing at the scale of $\sqrt{t_n}+(\log n/n)^{1/4}$ giving non-optimal bounds of that order for the bounded Lipschitz metric $\dd_{BL}(\mu_{t_n,n},\mu_0)$. Similarly, the Kolmogorov metric on $\C$ can be bounded using \cite[Theorem 5]{GJ} and \eqref{eq:quant_U_bound}, yielding $\dd_{K}(\mu_{t_n,n},\mu_0)\lesssim \big(t_n+\sqrt{\frac{\log n}{n}}\big)^{1/3}$ eventually, $\P$-a.s..
\end{remark}

\section{Repeated differentiation}\label{sec:diff}
We may now follow the same route for repeated differentiation.
Throughout this section, we fix $0<t<1$ and denote $Q_{t,n}(z):=\partial_z^{\lfloor tn\rfloor}P_n(z)$,  of degree $d_n:=n-\lfloor tn\rfloor$ and leading coefficient $n!/d_n!$. Recall its empirical distribution of roots is $\nu_{t,n}:=\frac{1}{d_n}\sum_{z:Q_{t,n}(z)=0}\delta_z$. For the logarithmic potential and Stieltjes transforms, we use the probability normalization
\begin{align}
	U_{t,n}(z):=\frac{1}{d_n}\log\left|\frac{d_n!}{n!}Q_{t,n}(z)\right|,
	\qquad
	m_{t,n}(z):=\frac{1}{d_n}\frac{Q'_{t,n}(z)}{Q_{t,n}(z)}.
	\label{eq:diff_U_m}
\end{align}
Here, we use the same notation as in Section \ref{sec:heat flow}, again satisfying $2\partial_zU_{t,n}=m_{t,n}$ and
$\partial_{\bar z}m_{t,n}=\pi\nu_{t,n}$ distributionally.

\subsection{Recursion formulas}
For $n\in\N$ sufficiently large, define the polynomials where the $k$-th factor is left out as
\begin{align}
	P_n^{[k]}(z):=\prod_{\substack{j\leq n\\j\neq k}}(z-z_j),
	\qquad
	Q_{t,n}^{[k]}:=\partial_z^{\lfloor tn\rfloor-1}P_n^{[k]},
 \qquad
	m_{t,n}^{[k]}(z):=\frac{1}{d_n}
	\frac{Q_{t,n}^{[k]\prime}(z)}{Q_{t,n}^{[k]}(z)}.
	\label{eq:diff_deleted}
\end{align}
Comparing to Lemma \ref{lem:commu}, here, deleting one factor is accompanied by taking one fewer derivative.
Consequently, $Q_{t,n}^{[k]}$ and $Q_{t,n}$ both have degree $d_n=n-\lfloor tn\rfloor$, and
$m_{t,n}^{[k]}$ has the probability normalization. This choice also makes
$m_{t,n}^{[k]}$ independent of $z_k$. 

\begin{prop}\label{prop:diff_crucial_rep}
	For 
	almost every $z$, we have
	\begin{align}
		m_{t,n}(z)
		&=\frac{1}{d_n}\sum_{k=1}^n
		\frac{1}{z-z_k+\dfrac{\lfloor tn\rfloor}{ {d_n}m_{t,n}^{[k]}(z)}},
		\label{eq:diff_crucial_rep_m}
		\\
				1&=\frac{1}{n}\sum_{k=1}^n
		\frac{1}{\frac{\lfloor tn\rfloor}{n}+\frac{d_n}{n}(z-z_k)m_{t,n}^{[k]}(z)}
		\label{eq:diff_crucial_rep_one}.
	\end{align} 
\end{prop}
All quotient identities are understood away from zeros of the denominators, or after cancellation as identities of rational functions. The following recursion- and commutation-identities follow immediately from the product rule $\partial_z^j(zf)=z\partial_z^jf+j\partial_z^{j-1}f$ for any polynomial $f$ and $j\in\N$.

\begin{lemma}\label{lem:diff_commu}
We have the recursions
	\begin{align}
		\sum_{k=1}^nQ_{t,n}^{[k]}
		&=\partial_z^{\lfloor tn\rfloor-1}\sum_{k=1}^nP_n^{[k]}
		=\partial_z^{\lfloor tn\rfloor-1}P'_n=Q_{t,n},
		\label{eq:diff_sum_deleted}
	\end{align}
	and, for every $1\leq k\leq n$,
	\begin{align}
		Q_{t,n}
		=(z-z_k)Q_{t,n}^{[k]\prime}+\lfloor tn\rfloor Q_{t,n}^{[k]}
		=Q_{t,n}^{[k]}
		\left({\lfloor tn\rfloor}+{d_n}(z-z_k)m_{t,n}^{[k]}(z)\right).
		\label{eq:diff_recursion}
	\end{align}
\end{lemma}

\begin{proof}[Proof of Proposition~\ref{prop:diff_crucial_rep}]
	Differentiate and divide \eqref{eq:diff_sum_deleted} by $d_nQ_{t,n}$, then apply
	\eqref{eq:diff_recursion} to get \eqref{eq:diff_crucial_rep_m}. Divide \eqref{eq:diff_sum_deleted} by 	$Q_{t,n}$ and use \eqref{eq:diff_recursion} to obtain \eqref{eq:diff_crucial_rep_one}.
\end{proof}

\begin{remark}
Similar statements hold for the logarithmic potential.
For $j=1,\ldots,\lfloor tn\rfloor$, we have
\begin{align*}
	\partial_z^jP_{d_n+j}(z)
	&=(z-z_{d_n+j})\partial_z^jP_{d_n+j-1}(z)
	+j\partial_z^{j-1}P_{d_n+j-1}(z)\\
	&=(d_n+j)\,\partial_z^{j-1}P_{d_n+j-1}
	\left(
	\frac{j}{d_n+j} 
	+\frac{z-z_{d_n+j}}{d_n+j}
	\frac{\partial_z^jP_{d_n+j-1}}
	{\partial_z^{j-1}P_{d_n+j-1}}
	\right).	
\end{align*}
After $\lfloor tn\rfloor$ iterations and taking the logarithm as in \eqref{eq:diff_U_m}, this leads to the recursive formula 
\begin{align}
	U_{t,n}(z)
 =\frac1{d_n}\log|P_{d_n}(z)|-\frac1{d_n}\log\binom{n}{\lfloor tn\rfloor}
+\frac1{d_n}\sum_{j=1}^{ \lfloor tn\rfloor}
\log\left|1
+\frac{z-z_{d_n+j}}{ j}
\frac{\partial_z^jP_{d_n+j-1}(z)}
{\partial_z^{j-1}P_{d_n+j-1}(z)}
\right|.
\label{eq:diff_crucial_rep_U}
\end{align}
\end{remark} 

\subsection{Convergence of the Stieltjes transform}

As in the heat-flow proof, leaving out one factor should not change the
limiting Stieltjes transform, thus we again expect $m_{t,n}^{[k]}\approx m_t(z)$
as in \eqref{eq:formal_m}. If so, then \eqref{eq:diff_crucial_rep_m} would translate to the limiting self-consistent equation, see below.
 
However, this requires attention near possible zeros
of $m_{t,n}^{[1]}\approx m_t$, which we additionally need to rule out by using the more continuous version \eqref{eq:diff_crucial_rep_one}. In this section, we shall work under the additional assumption of $\mu_0$ having a bounded density, which we shall remove in Section \ref{sec:removedensity}.
 
\begin{prop}[Limiting self-consistent equation]\label{prop:diff_m_t_limits}
Assume that $\mu_0$ has a bounded density, satisfies the $2+\delta$ logarithmic moment assumption \eqref{eq:log-ass} and
	\begin{align}\label{eq:ass_m0asymptotic}
		\lim_{|z|\to\infty}zm_0(z)=1.
	\end{align}
Then, for any $0<t<1$, every weakly convergent subsequence
	$\E\nu_{t,n}\Rightarrow\nu_t$ satisfies $m_t\neq0$ almost everywhere and
	\begin{align}
		m_t(z)
		=\frac 1 {1-t}m_0\left(z+\frac{t}{(1-t)m_t(z)}\right)
		\label{eq:diff_selfconseq} 
	\end{align}
in $L^1_{\mathrm{loc}}$. 
	If all such subsequential limits have the same Stieltjes transform
	$m_t$ on an open set $\mathcal D_t$, then
	$$
	m_{t,n}\ton m_t
	\qquad\text{almost surely in }L^1_{\mathrm{loc}}(\mathcal D_t).
	$$ 
\end{prop}
In particular, \eqref{eq:diff_selfconseq} implies $\|m_t\|_\infty\leq\|m_0\|_\infty/(1-t)$.  Assumption \eqref{eq:ass_m0asymptotic} will be necessary to continuously extend the self-consistent equation and is automatically satisfied in the case of $\mu_0$ being compactly supported or rotationally invariant.
For the proof, we need the following analogues of the heat-flow lemmas \ref{lem:leave_one_out} and \ref{lem:stieltjes_concentration}.

\begin{lemma}[Stability]\label{lem:diff_leave_one_out}
Assume that $\mu_0$ has bounded density. Let $K\subset\mathbb C$ be compact
, then
	\begin{align}
		\sup_{k\leq n}\E
		\left\|m_{t,n}-m_{t,n}^{[k]}\right\|_{L^1(K)} 
		\longrightarrow0.
		\label{eq:diff_lol_all_p}
	\end{align}
\end{lemma}

\begin{lemma}[Concentration of the Stieltjes transform]
	\label{lem:diff_stieltjes_concentration}
		Assume that $\mu_0$ has a bounded density and satisfies the $2+\delta$ logarithmic moment assumption \eqref{eq:log-ass}. Fix $t\ge 0$ and compact $K\subset\mathbb C$, then
\begin{align}
	\left\|m_{t,n}-\E m_{t,n}\right\|_{L^1(K)}\longrightarrow0
	\label{eq:diff_stieltjes_concentration}
\end{align}
	$\P$-almost surely and in $L^1(\P)$. Moreover, the statement remains true if $t_n\to0$ is a vanishing sequence depending on $n$.
\end{lemma}

This concentration of the Stieltjes transform, follows from the following concentration of the logarithmic potential exactly in the same way as Lemma \ref{lem:stieltjes_concentration} followed from Lemma \ref{lem:potential_concentration}.

\begin{lemma}[Concentration of the logarithmic potential]
	\label{lem:diff_potential_concentration}
		Assume that $\mu_0$ has a bounded density and satisfies the $2+\delta$ logarithmic moment assumption \eqref{eq:log-ass}.	
	Then, for each fixed $t\geq0$, the sequence $\E\nu_{t,n}$ is tight and, almost surely,
	\begin{align}
		U_{t,n}-\mathbb E U_{t,n}
		\longrightarrow0
		\qquad
		\text{in }L^1_{\mathrm{loc}}(\mathbb C)
		\label{eq:diff_potential_as_convergence}.
	\end{align}
	Moreover, the statement remains true if $t_n\to0$ is a vanishing sequence depending on $n$. If additionally, $\mu_0$ is compactly supported,  $K\subseteq\C$ is compact and  $0<\varepsilon<1$ is fixed, then there is a constant $c>0$, such that almost surely for sufficiently large $n$,
	\begin{align}\label{eq:jointwork}
		\sup_{0\le t<1-\varepsilon}\|U_{ t,n}-\mathbb EU_{ t,n}\|_{L^1(K)}\le c\sqrt{\frac{\log n}n}.
	\end{align}
\end{lemma}

Note in particular that the constant of the concentration inequality is uniformly bounded as $t_n\to 0$, since then $d_n\geq(1-t_n)n\sim n$. Again, let us first apply these Lemmas, before we prove them.

\begin{proof}[Proof of Proposition~\ref{prop:diff_m_t_limits}]
If $\mu_0$ were compactly suppported, then Gauss--Lucas theorem~\cite[Theorem 6.1]{Marden} implies that all our root measures $\nu_{t,n}$ are supported on the same compact set being the convex hull $ \operatorname{conv}(\operatorname{supp}\mu_0)$. More generally, tightness of $(\E\nu_{t,n})_n$ follows from Lemma \ref{lem:diff_potential_concentration}. Along a
	weakly convergent subsequence $\E\nu_{t,n}\Rightarrow\nu_t$ we have $\E m_{t,n}\to m_t$ in $L^1_{\mathrm{loc}}$. For any
	compact $K\subset\C$, Lemmas~\ref{lem:diff_leave_one_out} and
	\ref{lem:diff_stieltjes_concentration} give
	\begin{align*}
		\E\|m_{t,n}^{[1]}-m_t\|_{L^1(K)}
		&\leq\E\|m_{t,n}^{[1]}-m_{t,n}\|_{L^1(K)}
		+\E\|m_{t,n}-\E m_{t,n}\|_{L^1(K)}
		+\|\E m_{t,n}-m_t\|_{L^1(K)}\to0.
	\end{align*}
	Consequently, $\frac{d_n}{n} m_{t,n}^{[1]}\to(1-t) m_t$ in
	$L^1(\P\times \mathcal L|_K)$. 
	On the other hand, we want to average the recursion \eqref{eq:diff_crucial_rep_one} first over $z_k$ and then over the remaining i.i.d.~$z_j$, giving 
	\begin{align}
	1= \E\left[\frac{ m_0\Big(z+\frac{\lfloor tn\rfloor}{d_nm_{t,n}^{[1]}(z)} \Big)}{\frac{d_n}{n} m_{t,n}^{[1]}(z)}\right]\label{eq:Emo}
	\end{align}
for all $n \in\N$ and almost all $z\in\C$. To justify integrability, set $\tau=\lfloor tn\rfloor /n$ and $w=\frac{d_n}{n} m_{t,n}^{[1]}(z)$ and one can check that 
	$$\E\big[|\tau+(z-z_1)w|^{-1}|z_j,j\neq 1\big]\lesssim 1+\log(1+|w|^{-1})\lesssim 1+ \big|\log| m_{t,n}^{[1]}(z)|\big|\in L^1(\P\times \mathcal L |_K).$$
For the first inequality, note that bounded density and assumption \eqref{eq:ass_m0asymptotic} imply  $\E|a-z_1|^{-1}\lesssim\frac{1+\log(2+|a|)}{1+|a|}$ via small ball estimates, and the final integrability follows from  \eqref{eq:diff_crucial_rep_one}.
Assumption \eqref{eq:ass_m0asymptotic} also allows the continuous extension
\begin{align}
	\frac{m_0(z+\tau/w)}{w}
	=\frac1\tau+o(1),\qquad w\to0,\label{eq:context}
\end{align}
	uniformly for $z\in K$ and $\tau\in[t/2,2t]$. Since $m_0$ is bounded and uniformly continuous, the extended quotient tends
	uniformly to zero as $w\to\infty$ and becomes bounded and uniformly continuous in $(z,\tau,w)$ in their respective regions.
Using $\frac{\lfloor tn\rfloor}{d_n}\to t/(1-t)$ and completely analogous to \eqref{eq:eintollesDreieck}, we end up with
	\begin{align}
		\Bigg\|\frac{ m_0\Big(\cdot+\frac{t}{(1-t)m_{t }} \Big)}{(1-t)m_t}-1\Bigg\|_{L^1(K)}
		&\leq
		\E\Bigg\|
		\frac{ m_0\Big(\cdot+\frac{t}{(1-t)m_{t }} \Big)}{(1-t)m_t}
		- \frac{ m_0\Big(\cdot+\frac{\lfloor tn\rfloor}{d_nm_{t,n}^{[1]} } \Big)}{\frac{d_n}{n} m_{t,n}^{[1]} } 
		\Bigg\|_{L^1(K)}
		\longrightarrow0.\label{eq:nonotn}
	\end{align}
Crucially, on $\{z\in\C: m_t(z)=0\}$, the quotient equals $1/t\neq1$ by our assumption $0<t<1$.
	Hence $m_t\neq0$ almost everywhere, and the claimed self consistent equation \eqref{eq:diff_selfconseq} follows. 
	
	Finally, if all subsequential limits have the same Stieltjes transform on
	$\mathcal D_t$, compactness implies convergence of the full sequence
	$\E m_{t,n}\to m_t$ in $L^1_{\mathrm{loc}}(\mathcal D_t)$ 
	and Lemma~\ref{lem:diff_stieltjes_concentration} yields almost sure convergence. 
%
\end{proof}

\subsection{Removing the density assumption}\label{sec:removedensity}

\begin{prop}
	\label{prop:diff_regularization}
	Assume that $\mu_0$ satisfies \eqref{eq:asslocallog}.
	Let $0\le t_n\to t<1$ with $\sup_n t_n<1$, and define the regularized measures $\mu_0^\varepsilon
	:=\mu_0*\operatorname{Unif}(B_\varepsilon)	$.
	These measures have bounded densities, satisfy \eqref{eq:log-ass},	and converge weakly to $\mu_0$ as $\varepsilon\to0$.
	
	If, for every $\varepsilon>0$, the empirical root
	measures associated with i.i.d.\ roots of law $\mu_0^\varepsilon$ converge weakly almost surely $
	\nu_{t_n,n}^\varepsilon\Rightarrow\nu_t^\varepsilon
	$ to some deterministic distribution $\nu_t^\varepsilon$, then $\nu_t^\varepsilon\Rightarrow\nu_t$ as $\varepsilon\to0$ and we have almost sure weak convergence of the unregularized distributions
	$$
	\nu_{t_n,n}\Rightarrow\nu_t\qquad\text{ as }n\to\infty.
	$$
\end{prop}

The trick to remove the density assumption is simple: We independently perturb the roots and control the effect again by our recursive formulas. First, we shall control the perturbation in expectation, and then we shall make use of another concentration inequality which is tailored to our resampling strategy used throughout; The Efron-Stein inequality. The proof works exactly the same in the heat flow case as $t_n\to 0$, we omit the repetitive details.

\begin{proof}
	Define $z_k^\varepsilon:=z_k+\varepsilon X_k$, where the $X_k\sim \operatorname{Unif}(B_1)$ are independent circular law distributed random variables, independent of $z_j$.
	First, note that the logarithmic moment assumptions  \eqref{eq:asslocallog} together with \eqref{eq:somelogineq} give
	\begin{align}\label{eq:yetanotherlogmoment}
		C=\sup_{\substack{w\in\C\\0\le\varepsilon\le1}}
		\E\left|
		\log|w-z_1^\varepsilon|-\log(1+|w|)
		\right|^{2+\delta}<\infty.
	\end{align}
	For $\varepsilon>0$, this follows from $|\log(1+|a-\varepsilon X_1|)-\log(1+|a|)|\le\log2$.
	
	We use the recursion \eqref{eq:diff_recursion}, in which we replace the roots $z_k$ one at a time by $z_k^\varepsilon$, analogously to the truncation we have done in \eqref{eq:diff_Doob_increment} and \eqref{eq:Doob_increment}.
	Again, conditioning on all other roots, summing over the replacements, and using	$n/d_n\le(1-\sup_n t_n)^{-1}<\infty$, we obtain, for every compact $K\subseteq \C$,
	\begin{align*}
		\sup_n\E\|U_{t_n,n}-U_{t_n,n}^\varepsilon\|_{L^1(K)}
		&\lesssim 
		\sup_{w\in\C}
		\E\left|\log|w-z_1|-\log|w-z_1^\varepsilon|\right|\\
		&\le \sqrt{\varepsilon}
		+\frac{2C}{|\log\sqrt\varepsilon|^{1+\delta}}
		\to0,
	\end{align*}
	Now, truncating the log by $w\mapsto\log(|w|\vee\sqrt \varepsilon)$, which is a $\sqrt{1/\varepsilon}$-Lipschitz function at distance $\varepsilon$, we obtain 
	\begin{align*}\sup_{w\in\C}\E\left|\log|w-z_1|-\log|w-z_1^\varepsilon|\right| 	\le \sqrt{\varepsilon}+\mathrm{remainder}
	\end{align*}
	where the remainder consists of the two terms for $z_1^\varepsilon$ and $z_1$ of the form 
	\begin{align*}
		\sup_{w\in\C}\E\left[\log(|w-z_1^\varepsilon|\vee\sqrt\varepsilon)-\log|w-z_1^\varepsilon  |\right] \le\frac{C}{|\log\sqrt\varepsilon|^{1+\delta}}.
	\end{align*}
	Combining these, we have shown that uniformly in $n$ and as $\varepsilon\to 0$,
	$$ 
	\sup_n\E\|U_{t_n,n}-U_{t_n,n}^\varepsilon\|_{L^1(K)} 
	\lesssim \sqrt{\varepsilon} 
	+\frac{2C}{|\log\sqrt\varepsilon|^{1+\delta}} \to0, 
	$$ 
	When $\lfloor t_n n\rfloor=0$, the same estimate is trivial.
	
	Let us now argue how this implies $\nu_t^\varepsilon\Rightarrow\nu_t$ as $\varepsilon\to0$. By the distributional Poisson equation and letting first $n\to\infty$, we get in particular
	$$
	\left|\int\varphi \dd(\nu_t^\varepsilon-\nu_t^{\varepsilon'})\right|
	\lesssim
	\|\Delta\varphi\|_{L^1}
	\big(o_\varepsilon(1)+o_{\varepsilon'}(1)\big).
	$$
	for any $\varphi\in\mathcal C_c^\infty(\C)$ and as $\varepsilon,\varepsilon'\to0$. Thus, $\int\varphi\dd\nu_t^\varepsilon$ is a Cauchy sequence, hence convergent. Moreover, the a priori vague limit $\nu_t$ has full mass: Choose $\varphi_R(z):=\varphi(z/R)$ with $\|\Delta\varphi_R\|_{L^1}=\|\Delta\varphi\|_{L^1}$ for
	$\varphi\in\mathcal C_c^\infty(\C)$ with $0\le\varphi\le1$,
	$\varphi=1$ on $B_1$ and $\supp\varphi\subseteq B_2$, such that
	$$
	\nu_t^\varepsilon(B_{2R}^c)
	\le\int1-\varphi_R\dd\nu_t^\varepsilon\le
	1-\int\varphi_R\dd\nu_t^{\varepsilon'}
	+C\big(o_\varepsilon(1)+o_{\varepsilon'}(1)\big).
	$$
	Now send first $\varepsilon\to0$, then $R\to\infty$, and finally $\varepsilon'\to0$ proves asymptotic tightness, yielding the claimed weak convergence as $\varepsilon\to0$.
	
	Consequently, the distributional Poisson equation and the assumption $\nu_{t_n,n}^\varepsilon\Rightarrow\nu_t^\varepsilon$, first as $n\to\infty$ and then as $\varepsilon\to0$,
	give for all $\varphi\in\mathcal C_c^\infty(\C)$,
	$$
	\E\int\varphi\,\dd\nu_{t_n,n}
	\longrightarrow\int\varphi \dd\nu_t.
	$$
	It remains to establish almost sure convergence, for which we set up the resampling strategy.	For each $k$, let $\widehat\nu_{t_n,n}$ denote the measure obtained
	by replacing $z_k$ by an independent copy.
	Using the recursion \eqref{eq:diff_recursion} for yet another time, together with H\"older's inequality and the logarithmic moment bound \eqref{eq:yetanotherlogmoment} imply,
	$$
	\sup_{k\le n}
	\E\left|\int\varphi\,\dd(\nu_{t_n,n}-\widehat\nu_{t_n,n})
	\right|^{2+\delta}
	\le
	\frac{ C\,\|\Delta\varphi\|_{L^1}^{2+\delta}}
	{(\pi d_n)^{2+\delta}}
	\lesssim_\varphi n^{-2-\delta}
	$$
	for any $\varphi\in\mathcal C_c^\infty(\C)$. This serves as an input for the generalized Efron--Stein inequality
	\cite[Eq.~(1.3)]{Houdre} (see also \cite[Thm~15.5]{BLM13}), implying directly
	$$
	\E\left|\int\varphi\,\dd(\nu_{t_n,n}-\E\nu_{t_n,n})
	\right|^{2+\delta}
	\lesssim  n^{-1-\delta/2}.
	$$
	Markov's inequality and Borel--Cantelli now give almost sure
	convergence against any countable family of smooth
	compactly supported test functions, thus weak convergence to the probability distribution $\nu_t$.
\end{proof}

\subsection{Weak convergence}\label{sec:weakling}

Observe that contrary to the heat flow at $t_n\to 0$, we cannot use \eqref{eq:nonotn} because \eqref{eq:context} explodes and reciprocal norms $\| {m_\mu^{-1}} \|_{L^1(K)}$ cannot be easily controlled. In particular it may happen that $m_0(z)=0$ on a non-negligible set.  Hence, we run a variation of the argument using logarithmic potentials for the

\begin{proof}[Proof of Theorem \ref{thm:vanish}, continued] 
For the quantitative statement, we assume that $\mu_0$ has a bounded density on $\C$, and for the qualitative statement $\nu_{t_n,n}\Rightarrow\mu_0$ the same assumption is justified by Proposition \ref{prop:diff_regularization}. We need to control $\|\E U_{t_n,n}-U_0\|_{L^1(K)}\to0$ for any compact $K\subseteq \C$.
By \eqref{eq:diff_crucial_rep_U} and $U_0(z)=\E\frac1{d_n}\log|P_{d_n}(z)|$, we have 
	\begin{align}\label{eq:firststep}
	\E U_{t_n,n}(z)-U_0(z)
	=-\frac1{d_n}\log\binom{n}{\lfloor t_nn\rfloor}+ \frac1{d_n}\sum_{j=1}^{ \lfloor t_nn\rfloor} 
	\E \log\left|
	1
	+\frac{z-z_{d_n+j}}{j}
	\frac{\partial_z^jP_{d_n+j-1}(z)}
	{\partial_z^{j-1}P_{d_n+j-1}(z)}
	\right| , \end{align}
	where the first term is asymptotically negligible. 
	Recall that \eqref{eq:somelogineq} and the logarithmic moment assumptions \eqref{eq:asslocallog} give
	$$
	\sup_{a\in\mathbb C}
	\mathbb E\left|\log|a-z_1|-\log(1+|a|)\right|\le C.
	$$
For any $w\neq0$ and $a=z+w^{-1}$, we factor out $w$ to get
	\begin{align*}
		\mathbb E\left|\log|1+w(z-z_k)|\right|
		&\le
		\mathbb E\left|
		\log|z+w^{-1}-z_k|-\log(1+|z+w^{-1}|)
		\right|+\left|\log\bigl(|w|+|1+wz|\bigr)\right|\\
		&\le C_K+\log(1+C_K|w|),
	\end{align*}
	uniformly in $z\in K$.
Thus, we condition \eqref{eq:firststep} on the first $d_n+j-1$ roots, use independence of the roots $z_k$ and choose $w=\frac1j\frac{\partial_z^jP_{d_n+j-1}(z)}
{\partial_z^{j-1}P_{d_n+j-1}(z)}$.
	 Fubini's theorem, Jensen's inequality on $K$ and \eqref{eq:mLp} for the correct normalization $d_n$ then give
	\begin{align*}
		&\frac1{d_n}\sum_{j=1}^{\lfloor t_nn\rfloor}
		\E\bigg[\int_K
		\left|\log\big|
		1+\frac{z-z_{d_n+j}}{j}
		\frac{\partial_z^jP_{d_n+j-1}(z)}
		{\partial_z^{j-1}P_{d_n+j-1}(z)}
		\big|\right|\dd\mathcal L(z)\bigg]\\
		&\lesssim
		\frac{1}{d_n}
		\sum_{j=1}^{\lfloor t_nn\rfloor}\left( C_K +
		\log\left(1+\frac{C_K d_n}{j}\right)\right)
		\lesssim t_n+
		\frac{1}{d_n}\log\binom n{\lfloor t_nn\rfloor}
		\longrightarrow0.
	\end{align*}
	In total, we have proven
	\begin{align}
\|\E U_{t_n,n}-U_0\|_{L^1(K)}\lesssim t_n +\frac{1}{d_n}\log\binom n{\lfloor t_nn\rfloor}\lesssim t_n\log\big(t_n^{-1}\big),
	\end{align}
where $0\log(0^{-1})=0$. The concentration Lemma \ref{lem:diff_potential_concentration} yields $U_{t_n,n}\to U_0$ almost surely in $L^1(K)$. On the probability-1-event $\cap_{R\in\N} \{U_{t_n,n}\to U_0\text{ in }L^1(\overline B_R)\}$ we have almost sure weak convergence $\nu_{t_n,n}\Rightarrow\mu_0$.

 Moreover, if $\mu_0$ has compact support, then the quantitative bound \eqref{eq:jointwork} implies almost surely
\begin{align}
	\|U_{ t_n,n}-U_0\|_{L^1(K)}\lesssim t_n\log\big(t_n^{-1}\big)+\sqrt{\frac{\log n}n}
\end{align}
 for sufficiently large $n$. Now, we choose some test function $\varphi\in \mathcal C^2_c$ and apply the distributional Poisson  equation $\frac{1}{2\pi }\Delta (U_{ t_n,n}-U_0)=\nu_{ t_n,n}-\mu_0$. Then there exists some deterministic $c>0$, such that almost surely
 \begin{align}
\bigg|\int\varphi\dd\nu_{ t_n,n}-\int \varphi\dd \mu_0\bigg|\le \frac{c}{2\pi}\|\Delta\varphi\|_\infty \bigg(t_n\log\big(t_n^{-1}\big)+\sqrt{\frac{\log n}n}\bigg)
 	\end{align}
 	as claimed (with a new constant).
\end{proof}

\begin{remark} Let us provide a formal analogue of the heat flow integral identity of Remark \ref{rem:formalU}. 
Assuming sufficient joint convergence of all the Stieltjes transforms, we expect
	\begin{align*}
		U_t(z)
		&=U_0(z)+
		\int_0^t\left(
		U_0\Big(z+\frac{\tau}{(1-\tau)m_\tau(z)}\Big)
	+\log\left|(1-\tau)m_\tau(z)\right|\right)\frac{\dd\tau}{(1-\tau)^2}. 
	\end{align*}
which may also give another proof of Theorem \ref{thm:diff_radial}.
\end{remark}

We record a more general convergence theorem under an abstract uniqueness condition.

\begin{theorem}\label{thm:diff_main}
Assume that $\mu_0$ has a bounded density and satisfies \eqref{eq:log-ass} and \eqref{eq:ass_m0asymptotic}. Further suppose that there is at most one probability measure $\nu_t$ supported on the closure of the convex hull of $\operatorname{supp}\mu_0)$, whose Stieltjes transform satisfies the self-consistent equation
	\begin{align}
		1
		=\frac {m_0\left(z+\frac{t}{(1-t)m_t(z)}\right)} {(1-t)m_t(z)}
		\label{eq:diff_selfconseq2} 
	\end{align}
	where the equation is understood in $L^1_{\mathrm{loc}}$ and the right-hand side is extended to possible values $m_t(z)=0$ by continuity (see \eqref{eq:context}).
	Then, we have almost sure weak convergence
	$$
	\nu_{t,n}\Rightarrow\nu_t.
	$$
\end{theorem}

The uniqueness assumption in Theorem~\ref{thm:diff_main} is a separate non-trivial condition, which is not easily guaranteed as in the heat-flow case. However, rotationally invariant measures automatically select right inverse branch and give an unconditional result as forecast in Theorem \ref{thm:diff_radial}.  

\begin{proof}[Proof of Theorem~\ref{thm:diff_main}]
Lemma \ref{lem:diff_potential_concentration} gives tightness of $\E\nu_{t,n}$ and Proposition~\ref{prop:diff_m_t_limits} provide at least one probability measure $\nu_t$ satisfying \eqref{eq:diff_selfconseq2}. By the assumed
	uniqueness, every subsequential limit is this measure $\nu_t$.
	Hence $\E\nu_{t,n}\Rightarrow\nu_t$ weakly and
	$\E m_{t,n}\to m_t$ in $L^1_{\mathrm{loc}}$.
	Lemma~\ref{lem:diff_stieltjes_concentration} gives the almost sure
	convergence of $m_{t,n}$. Taking $\partial_{\bar z}$ distributionally gives almost sure weak convergence
	of $\nu_{t,n}$.
\end{proof}

In Section \ref{sec:intro} and \ref{sec:symm} below, all instances of the evolutions $\mu_0\Rightarrow\nu_t$ of Conjecture \ref{conj2} are rotationally invariant. Theorem \ref{thm:diff_main} also applies beyond rotationally invariant initial distributions, where limiting distributions remain absolutely continuous after the differentiation evolution.

\begin{remark}\label{rem:nonsymmex}
Let us discuss three possible classes of examples, in increasing  difficulty.

1) The most simple example is to take a symmetric $\tilde\mu_0$ and shift it to non-rotationally invariant $\mu_0:=\tilde\mu_0(\cdot-w)$ by any $w\in\C$, yielding $\nu_t=\tilde\nu_t(\cdot-w)$ since translation commutes with differentiation. However, this may not count  as genuinely non-rotationally invariant.

2) If we allow to fix $t>0$ first, then we may perturb a rotationally invariant density on a small region, which has been "differentiated away" up to that fixed time $t$. To be slightly more precise, one may take the circular law $\tilde \mu_0$, perturb it to $\mu_0:=\tilde\mu_0+\varepsilon\Delta h \dd\mathcal L$, where $h\in \mathcal C_c^\infty(B_{\sqrt t/2})$ is non-radial and real-valued. For $\varepsilon$ small enough, $\mu_0$ is a non-rotationally invariant distribution with bounded density, but after time $t<1$, the limit coincides with that of the circular law, i.e.
	$$
	\nu_{t,n}\Rightarrow
	\frac1{1-t}
	\left(w\mapsto w-\frac{t}{\bar w}\right)_\#
	\left(\tilde\mu_0|_{\{\sqrt t<|w|<1\}}\right).
	$$
To verify this, observe that the admissible of the two inverse branches of $S_t$ is unchanged, while the other invalid branch still has negative Jacobian, giving the uniqueness required by Theorem \ref{thm:diff_main}.

3) Constructing an example that is genuinely not rotationally invariant and remains absolutely continuous for all $0<t<1$ is non-trivial. Let us explain how this could be achieved.

Starting from the circular law $\tilde\mu_0$, we seek a diffeomorphism $F:B_1\to F(B_1)$ that deforms concentric circles into nested layers. The region between layers shall carry equal mass $\dd r$, and the layer $F(\{|z|^2=r\})$ should shrink to $r/10$, disappearing at time $r$. Thus, the collapse point $t/10$ moves with time, while the outer layers survive. The difficulty is to choose $F$ so that the Cauchy transform of $F_\#\tilde\mu_0$ generates precisely this motion, with the surviving layers remaining nested. An explicit choice (this was suggested by ChatGPT) is $F(z):=z e^{W_0(\bar z/10)}\bigl(1+W_0(\bar z/10)\bigr)$
where $W_0$ is the principal Lambert $W$-function. It can indeed be shown that $\mu_0:=F_\#\tilde\mu_0$ has bounded compactly supported density, is not rotationally invariant about any point, and, for every $0<t<1$, almost surely
$$
\nu_{t,n}\Rightarrow
\nu_t:=\frac1{1-t}(S_t)_\#
\bigl(\mu_0|_{F(\{\sqrt t<|z|<1\})}\bigr),
$$
where \(\nu_t\) is absolutely continuous. The remaining calculations are technical and contain no further insight for our study, so we omit them and continue with the more illuminating rotationally invariant framework below.
\end{remark}

\subsection{Proofs of the Lemmas}

\begin{proof}[Proof of Lemma~\ref{lem:diff_leave_one_out}]
	We follow the proof of Lemma \ref{lem:leave_one_out}, where one new estimate will be necessary to get the equivalent statement of \eqref{eq:lol_p0_estimate} with $d_n$ replacing the $n$.
	Logarithmic  differentiation of
	\eqref{eq:diff_recursion} from Lemma \ref{lem:diff_commu} gives the exact identity
	\begin{align}
		m_{t,n}-m_{t,n}^{[k]}
		= \frac{1
			+(z-z_k)\frac{\partial_zm_{t,n}^{[k]}}{m_{t,n}^{[k]}}}
		{ \frac{\lfloor tn\rfloor}{m_{t,n}^{[k]}}+ {d_n}(z-z_k)}
		\label{eq:diff_lol_difference}
	\end{align}
	for every sufficiently large $n\in\N$ and almost all $z\in\C$.
	Exactly as in \eqref{eq:Lploc_Stieltjes} for any $p<2$, the first $1$-summand is controlled by the uniform bound 
	\begin{align}
		\mathbb E\left[\int_K
		\left| 
		\frac{\lfloor tn\rfloor}{ m_{t,n}^{[k]}(z)}
		+d_n(z-z_k)
		\right|^{-p} \dd \mathcal L(z)\right]\lesssim d_n^{-p} \sup_{a\in\C}
		\int_{\C}|a-w|^{-p} \mu_0(dw)\lesssim d_n^{-p},
	\end{align}
	since $\mu_0$ has bounded density.
	
Now, fix $1/2<p_0<1$ so that by subadditivity it remains to control the term of \eqref{eq:diff_lol_difference} involving $\partial_zm_{t,n}^{[k]}$. First note that since $z\in K$ and assuming $|z_k|<R$ for some fixed $R>0$, one can improve the above by taking into account a possible dominating term ${\lfloor tn\rfloor}/{ m_{t,n}^{[k]}(z)}$, yielding
	\begin{align}
		 \mathbb E\left[
		\left|	 
		\frac{\lfloor tn\rfloor}{  m_{t,n}^{[k]}(z)}
		+d_n(z-z_k) 
		\right|^{-p_0}\mathbf1_{B_R}(z_k)
	 \middle|z_j,\ j\ne k
		\right]  \lesssim 
		\left(\frac{|m_{t,n}^{[k]}(z)|}{ \lfloor tn\rfloor}\right)^{p_0}\wedge d_n^{-p_0}
		\label{eq:diff_inverse_moment}
	\end{align}
with a constant depending on $R$.
As in the derivative estimate \eqref{eq:someLpbound} in the heat-flow proof, now with the normalization $d_n$ and $p_0<1$ give
\begin{align}
	\int_K
	|\partial_zm_{t,n}^{[k]}(z)|^{p_0}\dd\mathcal L(z)
	\lesssim d_n^{1-p_0}.\label{eq:delzmint}\end{align}
On the other hand, on the event $z_k\notin B_R$, we will use \eqref{eq:mLp} to get
\begin{align*} \E\left[\int_K\left|	m_{t,n}-m_{t,n}^{[k]}\right| ^{p_0}  \dd \mathcal L(z)\mathbf1_{B_R^c}(z_k)\right]
	&\le  C_{K,p_0} \mu_0(B_R^c).
\end{align*}
Finally, we use subadditivity in \eqref{eq:diff_lol_difference} and combine our collected bounds. First average over $z_k$ leading to cancellation of $m_{t,n}^{[k]}(z)$, then integrate out $z$ and afterwards average over $z_j,j\neq k$, yields uniformly in $k\le n$,
	\begin{align}
	 \mathbb E\left[ \int_K
		|m_{t,n}-m_{t,n}^{[k]}|^{p_0}\dd\mathcal L\right]
		&\le C_{K,p_0,R}\left( 
		d_n^{-p_0}
		+ \lfloor tn\rfloor^{-p_0}		\mathbb E\left[\mathbf1_{B_R}(z_k) \int_K
		|\partial_zm_{t,n}^{[k]}|^{p_0}\dd\mathcal L \right]\right)+C_{K,p_0}\mu_0(B_R^c)\nonumber \\
		&\le C_{K,p_0,R,t}\left( n^{-p_0}+ n^{1-2p_0}\right) +C_{K,p_0 } \mu_0(B_R^c),
		\label{eq:diff_lol_p0_estimate}
	\end{align}
where the last step follows from \eqref{eq:delzmint}. Taking the $\limsup$ for $n\to\infty$ bounds it by $\lesssim\mu_0(B_R^c)$.  Finally we take $R\to\infty$, and the extension to $p=1$ as claimed follows by the same H\"older inequality \eqref{eq:Hodor}.
	\end{proof}

\begin{proof}[Proof of Lemma \ref{lem:diff_potential_concentration}] We first follow the proof of Lemma \ref{lem:potential_concentration} almost exactly, hence we only mention the necessary changes, such as $d_n$ replacing some $n$'s. First, assume $\supp\mu_0\subseteq B_R$.
 	
Again let $\mathcal F_k:=\sigma(z_1,\ldots,z_k)$ and define the Doob martingale
	$M_k:=\E[U_{t,n}(z)\mid\mathcal F_k]$, which exists due to \eqref{eq:ULp}. Analogously to \eqref{eq:Doob_increment}, independence and the tower
	property imply
	\begin{align}
		M_k-M_{k-1}=\frac{1}{d_n}\E\Big[\log\Big|z+\frac{\lfloor tn\rfloor}{ {d_n} m_{t,n}^{[k]}(z)}-z_k\Big|-U_0\Big(z+\frac{\lfloor tn\rfloor}{ {d_n} m_{t,n}^{[k]}(z)}\Big)\mid\mathcal F_k\Big].
		\label{eq:diff_Doob_increment}
	\end{align} 
	Conditional Jensen's inequality, the tower property, and
	\eqref{eq:uniform_log_moments} now again give the moment Bernstein condition 
	\begin{align}
		\E\left[|M_k-M_{k-1}|^q\mid\mathcal F_{k-1}\right]
		\leq q!\left(\frac{C_R}{d_n}\right)^q,
		\qquad q\geq2
		\label{eq:diff_Bernstein_moments},
	\end{align} 
	which with $d_n\ge n\varepsilon$ implies exponential moments for  $0<\lambda<\frac{n\varepsilon}{C_R}$.
	\begin{align}
		\E\left[e^{\lambda |U_{t,n}(z) -\E U_{t,n}(z)|} \right]\le 2\exp \left( \frac{  n\lambda^2 C_R^2 }{d_n^2- d_n \lambda C_R}\right) 
		\le 2\exp \left( \frac{  \lambda^2 C_R^2 }{n\varepsilon^2- \varepsilon \lambda C_R}\right) .\label{eq:ExpMom}
	\end{align}
	Chernoff's bound with $\lambda=d_n^2u/(2nC_R^2+d_nC_Ru)$ yields the concentration inequality 
	\begin{align}
		\P\left(|U_{t,n}(z)-\E U_{t,n}(z)|>u\right)
		\leq2\exp\left(-\frac{d_n^2u^2}{2(2nC_R^2+d_nC_Ru)}\right).
		\label{eq:diff_potential_concentration}
	\end{align}
for some $C_R>0$ depending on the size of the support of $\mu_0$, for almost all $z\in\C$, every $u>0$ and all sufficiently large $n$. Note that these statements continue to hold if $\lfloor t n \rfloor=0$, despite \eqref{eq:diff_deleted} being ill defined. 

The common compact support required for the last arguments as in the proof of Lemma \ref{lem:potential_concentration} follows from Gauss--Lucas. The proof of \eqref{eq:jointwork} follows again as a joint work of our colleagues, with the only change that we take $c=5C_R\mathcal L(K)/\varepsilon$ and union bound of \eqref{eq:2or3} with exponent $3$, since now $U_{t,n}$ only depends on $\lfloor t n\rfloor$, which may take at most $n $ different values.
	
	It remains to remove the compact support assumption, for which we use the same truncation \eqref{eq:truncation}, which leads to the replacement of \eqref{eq:onesteptruncation} given by
	$$
	U_{t,n}^{(R,k-1)}(z)-U_{t,n}^{(R,k)}(z)
	=\frac1{d_n}
	\left(\log|a-z_k|-\log|a-z_k^{(R)}|\right),
	$$
	for $a=z+\frac{\lfloor tn\rfloor}{d_n\widetilde m_{t,n}^{[k]}(z)}$. All the remaining line of arguments of the proof of  Lemma \ref{lem:potential_concentration} follow analogously, among the \eqref{eq:Rbound}, tightness of $\E\nu_{t,n}$ via Gauss--Lucas Theorem and \eqref{eq:tightnessEmu}, as well as the choice $R_n=\exp(n^{1/(2+\delta)})$, leading inevitably to the claim.
	
\end{proof}

\subsection{The (local) push-forward}
The Jacobian calculation has the same form as for the heat flow, with
the additional factor $1/(1-t)$ from the probability normalization.
The inverse branch must first be selected by the convergence argument.

\begin{prop}[Local push-forward theorem]\label{prop:diff_localpushforward}Assume that $\mu_0$ has a bounded density and satisfies \eqref{eq:log-ass} and \eqref{eq:ass_m0asymptotic}. 
	Let $W_t\subset\{w:m_0(w)\neq0\}$ and $\mathcal D_t$ be open sets
	such that $S_t:W_t\to\mathcal D_t$ is a $\mathcal C^1$-diffeomorphism. If every subsequential weak
	limit $\nu_t$ of $\E\nu_{t,n}$ satisfies
	\begin{align}
		z+\frac{t}{(1-t)m_t(z)}\in W_t
		\qquad\text{for almost every }z\in\mathcal D_t,
		\label{eq:diff_branch_selection}
	\end{align}
	then, we have almost sure vague convergence 
	\begin{align}
		\nu_{t,n}|_{\mathcal D_t}\rightarrow \nu_t|_{\mathcal D_t}
		=\frac{1}{1-t}(S_t)_\#(\mu_0|_{W_t}).
		\label{eq:diff_localpushforward}
	\end{align}
	If also $\mu_0(W_t)=1-t$, then $\nu_{t,n}\Rightarrow\nu_t$
	almost surely globally.
\end{prop}

\begin{proof} 
	Recalling the definition \eqref{eq:diff_transport} of $S_t$, we see that Proposition~\ref{prop:diff_m_t_limits} implies $S_t^{-1}(z)= z+\frac {t}{(1-t)m_t(z)}$ on $\mathcal D_t$. Thus all subsequential
	transforms agree there, and Proposition~\ref{prop:diff_m_t_limits}
	proves $m_{t,n}(z)\to m_t(z)= 	\frac{1}{1-t}m_0(S_t^{-1}(z))$  almost surely in $L^p_{\mathrm{loc}}(\mathcal D_t)$, $0<p<2$. Now, we again have
\begin{align}
	\partial_{\bar z}m_t(S_t(w))
	=\frac{\partial_{\bar w}m_0(w)}{(1-t)\det JS_t(w)}\ge 0
	\label{eq:>0}\end{align}
	precisely as in the proof of Proposition \ref{prop:localpushforward}, and the rest of the proof follows identically.		
\end{proof}

\begin{remark}
	A diffeomorphism $S_t:W_t\to\mathcal D_t$ specifies its inverse branch, but does not establish that subsequential limits actually select this branch (which was automatic for the heat flow under the Lipschitz assumption). This is precisely the role of assumption \eqref{eq:diff_branch_selection}. For example, for the circular law,
	$m_0(w)=\bar w$ inside the unit disk. Thus, $\mathcal D_t=B_{1-t}\setminus\{0\}$, and $S_t$ can be inverted to either annulus $W_t=B_1\setminus \overline B_{\sqrt t}$ or $W_t=B_{\sqrt t}\setminus \overline B_t$. The radial argument in Section \ref{sec:symm} below selects the correct surviving region to be the outer annulus satisfying \eqref{eq:>0}.
\end{remark}

\subsection{Rotationally symmetric distributions}\label{sec:symm}
\begin{proof}[Proof of Theorem~\ref{thm:diff_radial}]
 Rotation of all input roots $z_j$ by some angle, rotates all zeros after repeated differentiation $Q_{t,n}$ by the same angle. Thus $\E\nu_{t,n}$ is rotationally invariant, and so is every
	subsequential limit $\nu_t$. Define the radial cumulative function as
	$$
	F_t(r):=(1-t)\nu_t(B_r ).
	$$
	The Stieltjes transform of a rotationally invariant measure satisfies
\begin{align}\label{eq:mtFt}
	(1-t)m_t(z)=\frac{F_t(|z|)}{z}
	\qquad\text{for } t\ge 0 \text{ and almost every }z\neq0.
\end{align}
Let us first consider the case where $\mu_0$ has a bounded density. Then, Proposition~\ref{prop:diff_m_t_limits} applies\footnote{One may also apply Theorem \ref{thm:diff_main} directly by arguing that any solution $m_t$ to \eqref{eq:diff_selfconseq2} satisfies $zm_t(z)\in\R$ and thus belongs automatically to a rotationally invariant measure.} since \eqref{eq:ass_m0asymptotic} always holds by \eqref{eq:mtFt} and it shows that $F_t(r)>0$ for almost every $r>0$. Since $F_t$ is nonnegative and nondecreasing, it cannot vanish. Thus, $r+\frac{tr}{F_t(r)}$ is well defined for any $r>0$, which we claim to be the radial inverse transport map. Substituting into the self-consistent equation \eqref{eq:diff_selfconseq} gives
	\begin{align}
		F_t(r)+t=F_0\left(r+\frac{tr}{F_t(r)}\right)>t
		\label{eq:diff_radial_inverse}
	\end{align}
	 for almost all $r>0 $.
	Hence, we may define the radial transport map $R_t(s):=s-\frac{st}{F_0(s)}$ satisfying $R_t(r+\frac{tr}{F_t(r)})=r$. Due to our density assumption of $\mu_0$, the function $F_0$ is continuous and nondecreasing, with
	$F_0(0)=0$ and $F_0(s)\to1$ as $s\to \infty$.
	On the interval $\{s>0:F_0(s)>t\}=:(s_t,\infty)$, the function $R_t$ is
	continuous and strictly increasing, since 
\begin{align}
	R_t(s_2)-R_t(s_1)
	\geq(s_2-s_1)\left(1-\frac{t}{F_0(s_1)}\right)>0,\qquad 0<s_t<s_1<s_2.\label{eq:RRR}
\end{align}
Moreover, $R_t(s_t)=0$ and $R_t(s)\sim(1-t)s$ as $s\to\infty$, hence the function $R_t:(s_t,\infty)\to(0,\infty)$ is a bijection with inverse $R_t^{-1}(r)=r+\frac{tr}{F_t(r)}$ and it is locally Lipschitz by \eqref{eq:RRR}.
Now \eqref{eq:diff_radial_inverse} holds for all $r>0$ by continuity, hence $F_t$ is uniquely defined and so is every radial subsequential limit $\nu_t$ of $\E\nu_{t,n}$. Again, as in the proof of Theorem \ref{thm:diff_main}, concentration from Lemma \ref{lem:diff_stieltjes_concentration} now yields the claimed almost sure vague (hence, weak) convergence of $\nu_{t,n}$.
	
Definition of $W_t=\{w\in\C:\mu_0(B_{|w|} )>t\}$ and continuity of $F_0$ gives $\mu_0(W_t)=1-t$.  
Combining the above, we obtain coinciding radial distribution function of the push-forward
\begin{align}\label{eq:Lipliplip}
(1-t)^{-1}(S_t)_\#(\mu_0|_{W_t})(B_r )=\frac {F_0(R_t^{-1}(r))-t}{1-t}=\nu_t(B_r ),\end{align}
similar to Proposition \ref{prop:diff_localpushforward}, where differentiability was additionally assumed. 

Most importantly, the inverses of \eqref{eq:diff_radial_inverse} or \eqref{eq:Lipliplip} are given by
\begin{align*}
	F_t^{-1}(u)=\frac{u}{u+t}F_0^{-1}(u+t),
	\qquad 0<u<1-t,
\end{align*}
which is the claimed representation \eqref{eq:quantiles}.

Let us now remove the bounded density assumption, for which we assume now that $\mu_0$ is rotationally invariant satisfying only \eqref{eq:ass_rot}. Then, the uniform local logarithmic moment assumption \eqref{eq:asslocallog} follows already from our assumption \eqref{eq:ass_rot}: Using $\bigl||z|-re^{i\theta}\bigr|\ge r|\sin\theta|$ and an angular integration, we obtain
 \begin{align*}
 \sup_{z\in\mathbb C}\int_{\mathbb C}
	|\log(1\wedge|z-w|)|^{2+\delta} \mu_0(\dd w)
  &\lesssim
	\int_{\mathbb C}|\log(1\wedge|w|)|^{2+\delta} \mu_0( \dd w)
	+\frac1{2\pi}\int_0^{2\pi}
	|\log|\sin\theta||^{2+\delta}\dd\theta \\
 & \lesssim\int_{\C}|\log|w||^{2+\delta}\mu_0(\dd w)+1
	<\infty.
	\end{align*} 
Therefore, we may call on the regularization $z_k^\varepsilon=z_k+\varepsilon X_k$ from Proposition \ref{prop:diff_regularization} for fixed $t$ and the regularized distribution $\mu_0^\varepsilon :=\mu_0*\operatorname{Unif}(B_\varepsilon)$ with bounded density. 
The bounded-density case gives almost sure weak convergence $\nu_{t,n}^\varepsilon\Rightarrow\nu_t^\varepsilon$, with convergent quantiles 
$$
\sup_{0<u<1}
\left|(F_0^\varepsilon)^{-1}(u)-F_0^{-1}(u)\right|
\le\varepsilon.
$$
Proposition \ref{prop:diff_regularization} then shows $\nu_{t,n}\Rightarrow\nu_t$ almost surely with the same quantile representation \eqref{eq:quantiles} as $\varepsilon\to 0$. In particular $s_t:=\inf\{r>0:F_0(r)>t\}=F_0^{-1}(t+)>0$ by \eqref{eq:ass_rot} and $F_t^{-1}(u)>0$ for every $u>0$, so $\nu_t$ has no atom at the origin. Furthermore, for any $0<u<v<1-t$, we obtain
\begin{align*}
	F_t^{-1}(v)-F_t^{-1}(u)
	\ge
	F_0^{-1}(u+t)
	\left(\frac{v}{v+t}-\frac{u}{u+t}\right)
	\ge
	\frac{t s_t(v-u)}{(u+t)(v+t)}
	\ge t s_t(v-u).
\end{align*} 
Hence $F_t$ is Lipschitz with constant $(t s_t)^{-1}<\infty$. By rotational invariance, $\nu_t$ has a density, for almost every $z\neq0$ satisfying
$$
\frac{\dd\nu_t}{\dd\mathcal L}(z)
=\frac{F_t'(|z|)}{2\pi(1-t)|z|}\le \frac{1}{2\pi(1-t)t s_t|z|} .
$$
\end{proof}

\begin{proof}[Proof of Corollary~\ref{cor:diff_circ}]
	Here $F_0(s)=s^2$ for $0\leq s\leq1$, and the surviving roots
	start in the annulus $\sqrt t<|w|<1$. On this annulus,
	$$
	S_t(se^{i\theta})=\left(s-\frac{t}{s}\right)e^{i\theta},
	$$
	and $ R_t(s)=s -\frac{t}{s } $.
	For $0<r<1-t$, its 	radial inverse is  $r\mapsto \frac{r+\sqrt{r^2+4t}}{2}$
and we obtain
$$\nu_t(B_r )
	=\frac{r\bigl(r+\sqrt{r^2+4t}\bigr)}{2(1-t)},
	\qquad 0\leq r\leq1-t, $$
Differentiation with respect to $r$, and dividing by $2\pi r$, gives the density as claimed in \eqref{eq:diff_circ_density}. 
\end{proof}

\begin{example}
Let us finish with a slightly more general example that exhibits a striking difference between heat flow and repeated differentiation.
For $0\le r<1$, consider the uniform distribution  $\mu_0$ on the annulus $\{z\in\C:r\le |z|\le 1\}$. Then, the same arguments as for Corollary \ref{cor:circ} and Corollary \ref{cor:diff_circ} show
\begin{align}	
\supp\mu_t = \left\{ x+iy: \frac{x^2}{(1+t)^2} +\frac{y^2}{(1-t)^2}\le1,\quad x^2+y^2\ge r^2 \right\},\quad t<(1-r^2)/2
\end{align}
where $m_0$ is $L=2/(1-r^2)$-Lipschitz.
On the other hand,  $ R_t(s)=s\left(1-\frac{t(1-r^2)}{s^2-r^2}\right)$ gives
\begin{align}
\supp\nu_t =\{|z|\le (1-t)\}, \quad 0<t<1.
\end{align} 
Hence, for small time, the heat flow preserves the hole $B_r $, while repeated differentiation immediately fills it.
\end{example}
 \AtEndDocument{\vfill 
 	\begin{center}
 		\emph{AI steals all my fun out of the challenge.} ------ JJ \\
 		\emph{JJ steals all my lemmas out of the chat.} \rlap{\raisebox{.58ex}{\scalebox{.06}{Actual AI reply: “I’m here to help with the proof, not spoil the discovery.}}}{------} AI
 	\end{center}
 }
%

\end{document}